\documentclass[a4paper,11pt,reqno]{amsart}

\usepackage{amssymb}
\usepackage{epsfig}
\usepackage{amsfonts}
\usepackage{amsmath}
\usepackage{euscript}
\usepackage{amscd}
\usepackage{amsthm}
\DeclareMathAlphabet{\mathpzc}{OT1}{pzc}{m}{it}
\usepackage{color}
\usepackage{mathtools}

\usepackage{mathrsfs}

\DeclareFontFamily{OT1}{pzc}{}
\DeclareFontShape{OT1}{pzc}{m}{it}{<-> s * [1.10] pzcmi7t}{}
\DeclareMathAlphabet{\mathpzc}{OT1}{pzc}{m}{it}

\DeclareSymbolFont{SY}{U}{psy}{m}{n}
\DeclareMathSymbol{\emptyset}{\mathord}{SY}{'306}

\theoremstyle{plain}

\newtheorem{thm}{Theorem}[section]
\newtheorem*{thm*}{Theorem}
\newtheorem{cor}[thm]{Corollary}
\newtheorem{lem}[thm]{Lemma}
\newtheorem{prop}[thm]{Proposition}
\newtheorem{defn}[thm]{Definition}
\newtheorem{rem}[thm]{Remark}

\numberwithin{equation}{section}

\def\C{{\mathbb C}}

\def\v{\varphi}
\def\l{\lambda}

\def\ra{\rightarrow}
\def\ov{\overline}
\def\lo{\longrightarrow}

\def\m{\mathcal}
\def\mb{\mathbb}
\def\mr{\mathrm}
\def\a{\alpha}
\def\b{\beta}
\def\g{\gamma}
\def\d{\sum}

\def\wi{\widetilde}

\def\w{\widehat }

\def\beq{\begin{eqnarray}}
\def\eeq{\end{eqnarray}}
\def\beqa{\begin{eqnarray*}}
\def\eeqa{\end{eqnarray*}}

\def\ov{\overline}
\def\bl{\boldsymbol}

\def\i{\prime}

\newcommand{\be}{\begin{equation}}
\newcommand{\ee}{\end{equation}}
\newcommand{\bea}{\begin{eqnarray}}
\newcommand{\eea}{\end{eqnarray}}
\newcommand{\Bea}{\begin{eqnarray*}}
\newcommand{\Eea}{\end{eqnarray*}}
\newcommand{\inner}[2]{\langle #1,#2 \rangle }%

\newcounter{cnt1}
\newcounter{cnt2}
\newcounter{cnt3}
\newcommand{\blr}{\begin{list}{$($\roman{cnt1}$)$}
 {\usecounter{cnt1} \setlength{\topsep}{0pt}
 \setlength{\itemsep}{0pt}}}
\newcommand{\bla}{\begin{list}{$($\alph{cnt2}$)$}
 {\usecounter{cnt2} \setlength{\topsep}{0pt}
 \setlength{\itemsep}{0pt}}}
\newcommand{\bln}{\begin{list}{$($\arabic{cnt3}$)$}
 {\usecounter{cnt3} \setlength{\topsep}{0pt}
 \setlength{\itemsep}{0pt}}}
\newcommand{\el}{\end{list}}

\begin{document}
\title[reducing submodules]{Reducing sub-modules of the Bergman module $\mathbb A^{(\lambda)}(\mathbb D^n)$\\ under the action of the symmetric group}
\author[Biswas]{Shibananda Biswas}
\author[Ghosh]{Gargi Ghosh}
\address[Biswas, Ghosh and Shyam Roy]{Indian Institute of Science Education and Research Kolkata, Mohanpur 741246, Nadia, West Bengal, India}
\email[Biswas]{shibananda@gmail.com}
\email[Ghosh]{gg13ip034@iiserkol.ac.in}
\email[Shyam Roy]{ssroy@iiserkol.ac.in}
\author[Misra]{Gadadhar Misra}
\address[Misra]{Department of Mathematics, Indian Institute of Science, Banaglore 560012, India}
\email[Misra]{gm@math.iisc.ernet.in}
\author[Shyam Roy]{Subrata Shyam Roy}

\thanks{The work of S. Biswas is partially supported by Inspire Faculty Fellowship (IFA-11MA-06)
funded by DST, India at IISER Kolkata. The research of G. Misra was supported, in part, by a grant of the project MODULI under the IRSES Network, the J C Bose National Fellowship and the UGC, SAP -- CAS}

\begin{abstract}
The weighted Bergman spaces on the polydisc, $\mb A^{(\lambda)}(\mb D^n)$, $\lambda>0,$
splits into orthogonal direct sum of subspaces $\mb P_{\bl p}\big(\mb A^{(\l)}(\mb D^n)\big)$ indexed by the partitions $\bl p$ of $n,$
which are in one to one correspondence with the equivalence classes of the irreducible representations of the symmetric group on $n$ symbols.
In this paper, we  prove that each sub-module $\mb P_{\bl p}\big(\mb A^{(\l)}(\mb D^n)\big)$ is a \emph{locally free Hilbert module} of \emph{rank} equal to square of the dimension $\chi_{\bl p}(1)$ of the corresponding irreducible representation.
Given two partitions $\bl p$ and $\bl q$, we show  that  if  $\chi_{\bl p}(1) \ne \chi_{\bl q}(1),$ then the sub-modules
$\mathbb P_{\bl p}\big (\mb A^{(\lambda)}(\mb D^n)\big )$ and $\mathbb P_{\bl q}\big (\mb A^{(\lambda)}(\mb D^n)\big )$ are not equivalent. For the trivial and the sign representation corresponding to the partitions  $\bl p = (n)$ and $\bl p = (1,\ldots,1)$, respectively, we prove that the sub-modules $\mb P_{(n)}\big(\mb A^{(\l)}(\mb D^n)\big)$ and $\mb P_{(1,\ldots,1)}\big(\mb A^{(\l)}(\mb D^n)\big)$ are inequivalent. In particular, for $n=3$, we show that all the sub-modules in this decomposition are inequivalent.
\end{abstract}

\maketitle
\section{Introduction}

In this paper, we study the weighted Bergman space $\mb A^{(\l)}(\mb D^n),$ $\lambda > 1,$ of square integrable holomorphic functions defined on the polydisc $\mb D^n$ with respect to the measure $\big ( \prod_{i=1}^n(1-|z_i|^2)^{\l -2} \big ) dV(\bl z),$ $\bl z\in \mb D^n.$ (In the sequel, we also consider the case of $\l >0.$)
The bi-holomorphic automorphism group $\mr{Aut}(\mb D^n)$ is easily seen to be the
semi-direct product ${\mr{Aut}(\mb D)}^n\rtimes\mathfrak S_n,$ where $\mathfrak S_n$ is the permutation group on
$n$ symbols. For $\Phi\in\mr{Aut}(\mb D^n),$ define $U:{\mr{Aut}(\mb D^n)}\to \m L\big(\mb A^{(\l)}(\mb D^n)\big)$
by the formula:
$$U(\Phi^{-1}) h=\big(\det (D\Phi)\big)^{\l/2}h\circ \Phi,\,\, h\in\mb A^{(\l)}(\mb D^n).$$

Since $\big(\det (D\Phi)\big)^{\l/2}(\bl z):\mr{Aut}(\mb D^n)\times \mb D^n \to \mathbb C$ is
a (projective) cocycle,  it follows that the map $U$ defines a (projective) unitary representation. The Hilbert space $\mb A^{(\l)}(\mb D^n)$ is also a module over the polynomial ring $\C[\bl z],$ namely,
\be \label{module}
m_p(h) = p \cdot h,\,\, p \in  \C[\bl z],\,\, h \in \mb A^{(\l)}(\mb D^n),\nonumber\ee
where $p\cdot h$ is the point-wise multiplication.
Setting  $(\Phi \cdot f)(\bl z) = f\big (\Phi^{-1}(\bl{z}) \big ),$ we have the relationship
$m_{\Phi \cdot p} = U(\Phi)^* \, m_p\, U(\Phi),\,\, \Phi\in \mr{Aut}(\mb D^n), p \in \mathbb C[\bl z],$
which is analogous to the imprimitivity introduced by Mackey (cf. \cite[Chapter 6]{VSV}).
The imprimitivities of Mackey have been studied extensively  and are related to induced representations, representations of the semi-direct product and homogeneous vector bundles, see Theorems 6.12, 6.20 and 6.24 in \cite{VSV}, respectively.  However, the situation we have described
is different in that the module action is defined over the ring of analytic polynomials rather than the algebra of continuous functions. This, we believe, merits a detailed investigation and the outcome, see \cite{KM, MU}, so far is very encouraging.
Also, the restriction of the representation $U$
to the  subgroup $\triangle:=\{(\varphi, \ldots , \varphi): \varphi \in \mbox{\rm Aut}(\mathbb D) \}$ of $\mr{Aut}(\mb D^n)$  has a decomposition into irreducible components  known as the Clebsch-Gordan decomposition.  On the other hand, the symmetric group acts on $\mb A^{(\l)}(\mb D^n)$ via
the unitary map $R_{\sigma^{-1}}: h \to h\circ \sigma,\,\sigma \in \mathfrak S_n.$ The Hilbert space $\mb A^{(\l)}(\mb D^n)$
is also a module over the ring  of the symmetric polynomials $\C[\bl z]^{\mathfrak S_n},$ where the module map is given by the formula:
$\mathfrak m_p(h) = p \cdot h,\,\,p\in \C[\bl z]^{\mathfrak S_n}.$
Here, we propose to study the imprimitivity $\big (\mb A^{(\l)}(\mb D^n),  \mathfrak m_p, R_\sigma \big )$ and obtain a decomposition of the Hilbert module $\mb A^{(\l)}(\mb D^n)$ into sub-modules like in the more familiar Clebsch-Gordan decomposition mentioned above.

Let $\widehat{\mathfrak S_n}$ denote the equivalence classes of all irreducible representations of $\mathfrak S_n.$ 
It is well known that these are finite dimensional and they
are in one to one correspondence with partitions $\bl p$ of $n$ \cite[Theorem 4.3]{FH}. Recall that a partition $\bl p$ of $n$ is a decreasing finite sequence $\bl p = (p_1,\ldots,p_k)$ of non-negative integers  such that $\sum_{i=1}^k p_i = n.$ A partition $\bl p$ of $n$ is denoted by $\bl p\vdash n$. Let $\bl\pi_{\bl p}$ be a unitary representation of $\mathfrak S_n$ in the equivalence class of $\bl p\vdash n$, that is, $\bl\pi_{\bl p}(\sigma) = \big (\!\big (\bl\pi_{\bl p}^{ij}(\sigma)\big )\!\big )_{i,j = 1}^{m} \in \C^{m\times m}$, $\sigma \in \mathfrak S_n,$ where $m = \chi_{\bl p}(1)$
and $\chi_{\bl p} (\sigma)= \mbox{\rm trace}\big (\bl \pi_{\bl p} (\sigma)\big ),$ $\sigma \in \mathfrak S_n,$  is the character of the representation $\bl\pi_{\bl p}.$

A decomposition of $\mb A^{(\l)}(\mb D^n)$ under the natural action of the group ${\mathfrak S_n},$ which is the restriction of $U$ to $\mathfrak S_n,$ is given by the formula (cf. \cite{BS} and \cite{MSZ}):
\Bea \label{C-G}
\mb A^{(\l)}(\mb D^n) = \bigoplus_{\bl p\,\vdash\, n}  \mathbb P_{\bl p}\,\big (\mb A^{(\l)}(\mb D^n)\big ),\,\, \bl p \vdash n\,\, \mbox{\rm is a partition of}\,\, n\in \mathbb N, \Eea
where
$\mathbb{P}_{\bl p} f  =
\frac{\chi_{\bl p}(1)}{n!}\sum_{\sigma\in\mathfrak S_n}\ov{\chi_{\bl p}(\sigma)} (f \circ \sigma^{-1}), \sigma \in \mathfrak S_n.$
On the right hand side, the irreducible representation of the group $\mathfrak S_n$ corresponding to the partition $\bl p$ is not multiplicity free.  Both sides of the equation \eqref{C-G} happen to be modules over  $\C[\bl z]^{\mathfrak S_n},$ what is more, the explicit projection formula has been used extensively in \cite{BS} to study various properties of the Hilbert module $\mathbb P_{\bl p} \big ( \mb A^{(\l)}(\mb D^n)\big ).$


For the sake of concreteness, we have picked
the Hilbert module $\mb A^{(\l)}(\mb D^n)$ over the ring $\C[\bl z]^{\mathfrak S_n},$ however, the questions we raise here can be made up in similar but much more general context.

Let $K$ be a $\mathfrak S_n$-invariant positive definite  kernel on $\mathbb D^n$ and
$\mathcal H_K$ be the corresponding reproducing kernel Hilbert space.
Let $\bl\pi_{\bl p}$ be the matrix representation of the finite dimensional unitary representation of $\mathfrak S_n$ corresponding to the partition $\bl p\vdash n.$
Define the operators $\mathbb{P}_{\bl p}^{ij}: \mathcal H_K\ra\mathcal H_K,\, 1\leq i, j\leq \chi_{\bl p}(1),$ by the formula
$$\mathbb{P}_{\bl p}^{ij} f = \frac{\chi_{\bl p}(1)}{n!}\sum_{\sigma\in\mathfrak{S}_n}\bl\pi_{\bl p}^{ji}(\sigma^{-1})(f \circ \sigma^{-1}).
$$
Also, $\mathbb{P}_{\bl p} = \sum_{i = 1}^{\chi_{\bl p}(1)} \mathbb{P}_{\bl p}^{ii}$. Specializing to our situation, that is, when $K(\bl z , \bl w) =  \prod_{i=1}^n(1-z_i \bar{w}_i)^{-\l }$ and $\mathcal H_K= \mathbb A^{(\l)} (\mathbb D^n),$ we ask
\begin{enumerate}
\item if the sub-modules $\mb P_{\bl p}\big(\mb A^{(\l)}(\mb D^n)\big)$
and $\mb P_{\bl q}\big(\mb A^{(\l)}(\mb D^n)\big)$ are inequivalent  for distinct partitions $\bl p$ and $\bl q$ of $n;$
\item if the reducing sub-modules $\mb P_{\bl p}^{ii}\big(\mb A^{(\l)}(\mb D^n)\big)$ and $\mb P_{\bl q}^{jj}\big(\mb A^{(\l)}(\mb D^n)\big)$ are inequivalent whenever $(\bl p, i)\ne (\bl q, j),$ where $\bl p,\, \bl q$ are partitions of $n,$  $1\leq i \leq \chi_{\bl p}(1)$ and $1\leq j \leq \chi_{\bl q}(1),$
\item if the reducing sub-modules $\mb P_{\bl p}^{ii}\big(\mb A^{(\l)}(\mb D^n)\big),$ $\bl p$ partition of $n$ and $1\leq i \leq \chi_{\bl p}(1),$ are minimal?
\end{enumerate}

\noindent For any partition $\bl p$ of $n,$ we have shown, see Corollary \ref{cor2.13}, that
the Hilbert modules $\mb P_{\bl p}\big(\mb A^{(\l)}(\mb D^n)\big)$  are locally free of rank $\chi_{\bl p}(1)^2$ on an open subset of $\mb G_n.$
Furthermore, using similar arguments, we show that the  sub-modules $\mb P_{\bl p}^{ii}\big(\mb A^{(\l)}(\mb D^n)\big),$ $1\leq i \leq \chi_{\bl p}(1),$ are locally free of rank $\chi_{\bl p}(1).$
%
%
Therefore, if $\chi_{\bl p}(1) \neq \chi_{\bl q}(1)$, then the sub-modules $\mb P_{\bl p}^{ii}\big(\mb A^{(\l)}(\mb D^n)\big)$ and $\mb P_{\bl q}^{jj}\big(\mb A^{(\l)}(\mb D^n)\big)$ are not equivalent, see Theorem \ref{m1}.
Although, we haven't been able to resolve this issue when $\chi_{\bl p}(1) = \chi_{\bl q}(1),$ in general, we have obtained the answer in one important special case, namely, for all partition $\bl p$ of $n$ such that $\chi_{\bl p}(1) =1.$ For  $n\geq 2,$ there are only two such partitions: $\bl p= (n)$ or $(1,\ldots , 1).$ We show that the two  sub-modules  $\mb P_{(n)}\big (\mb A^{(\l)}(\mb D^n)\big )$ and $\mb P_{(1,\ldots,1)}\big (\mb A^{(\l)}(\mb D^n)\big )$ are  inequivalent (there is no intertwining module map between them that is unitary)  over $\C[\bl z]^{\mathfrak S_n},$ see Theorem \ref{m}.
Also these summands are locally free of rank $1,$ therefore they are irreducible and hence minimal.
For $n=2,$ in the decomposition $\mathbb A^{(\l)}(\mathbb D^2) = \mb P_{(2)}\big ( \mb A^{(\l)}(\mb D^2) \big )\oplus \mb P_{(1,1)}\big (\mb A^{(\l)}(\mb D^2)\big ),$ the two summands are minimal and inequivalent. Therefore, in this case, we have answered the questions (1) - (3).
Furthermore, for $n=3$, it follows that all the submodules in the decomposition
$\oplus_{\bl p\,\vdash\, 3}\mathbb P_{\bl p}\big ( \mb A^{(\l)}(\mb D^3)\big )$
are inequivalent, see Corollary \ref{corf}. Along the way, we give an explicit formula, see Theorem \ref{BergKerGn}, for the weighted Bergman kernel of the symmetrized polydisc $\mathbb G_n$ in the co-ordinates of $\mathbb G_n$ rather than that of the polydisc $\mathbb D^n.$  In an earlier paper \cite{MSZ}, the case of $n=2$ was worked out.


For any partition $\bl p$ of $n,$ we recall from \cite{BS} that the commuting $n$-tuple of multiplications $M^{(\bl p)}_{\bl s} = (M_{s_1}, \ldots , M_{s_n})$ by the elementary symmetric functions $\bl s$  defined on the Hilbert space $\mb P_{\bl p}\big(\mb A^{(\l)}(\mb D^n)\big),$  $\l \geq 1,$ are examples of $\Gamma_n$-contractions. Since  $\mb P_{\bl p}\big(\mb A^{(\l)}(\mb D^n)\big)$ admits a further decomposition into a direct sum of the sub-modules $\mb P_{\bl p}^{ii}\big(\mb A^{(\l)}(\mb D^n)\big),$ $1\leq i \leq \chi_{\bl p}(1),$ it follows that the $n$-tuple  $M^{(\bl p)}_{\bl s}$ acting on these reducing subspaces is also a $\Gamma_n$-contraction, which is Theorem \ref{thmB} of this paper. What is more, we have shown that the Taylor  joint  spectrum of each of these $n$-tuples  is $\Gamma_n$ and thus, in these examples, the spectrum is a spectral set.

Since the Hilbert module $\mb P_{\bl p}\big(\mb A^{(\l)}(\mb D^n)\big),$ as well as the  sub-modules $\mb P_{\bl p}^{ii}\big(\mb A^{(\l)}(\mb D^n)\big),$ $1\leq i \leq \chi_{\bl p}(1)$, are locally free on some open subset of $\mb G_n,$
it follows that these are in one to one correspondence with holomorphic hermitian vector bundles defined on some open subset of $\mb G_n.$
The rank of this vector bundle is an invariant, albeit a very weak one. However, it is the rank  which is used to distinguish the sub-modules $\mb P_{\bl p}\big(\mb A^{(\l)}(\mb D^n)\big)$ in this paper.
We conclude the paper with an explicit realization of a spanning holomorphic cross-section for the sub-modules $\mb P_{\bl p}^{ii}\big(\mb A^{(\l)}(\mb D^n)\big).$ This provides an invariant that we believe will be useful in our future work.

\subsection*{\sf Acknoledgement}The research of Biswas, Misra and
Shyam Roy was partially supported through the programme ``Research in Pairs''  by the Mathematisches Forschungsinstitut Oberwolfach in 2016. It was completed during their short visit to ICMAT, Madrid in 2017. We thank both these institutions for their hospitality.
\section{Locally free Hilbert modules}
First, we  recall several useful definitions following \cite{DP, CG} and \cite{CD}.
\begin{defn}
A Hilbert space $\mathcal H$ is said to be a Hilbert module over the polynomial ring $\mathbb C[\bl{z}]$ in $n$ variables if the map $(p,h) \to p\cdot h,$ $p\in \mathbb C[\bl{z}], h\in \mathcal H,$ defines a homomorphism $p \mapsto T_p,$ where $T_p$ is bounded operator defined by  $T_p h = p \cdot h.$

Two Hilbert modules $\mathcal H$ and $\tilde{\mathcal H}$ are said to be (unitarily)  equivalent if there exists a unitary module map $U:\mathcal H \to \tilde{\mathcal H},$ that is, $UT_p = \tilde{T_p} U,\, p \in \mathbb C[\bl{z}].$
\end{defn}
Let $\C_{\bl w}$ be the one dimensional module over the polynomial ring $\mathbb C[\bl{z}]$ defined by the evaluation, that is, $(p, c) \to p(\bl w) c,\,\, c\in\mathbb C, p\in  \mathbb C[\bl{z}].$
Following \cite{DP}, we define the module tensor product of two Hilbert modules $\m H$ and $\C_{\bl w}$ over  $\mathbb C[\bl{z}]$ to be the quotient of the
space Hilbert space tensor product $\m H\otimes \mb C_{\bl w}$ by the subspace
\beqa
\mathcal N&:=& {\{p\cdot f\otimes 1_{\bl w} - f\otimes p(\bl w): p\in\mb C[\bl z] , f\in\m H\}}\\
&=& {\{(p - p(\bl w))f : p\in  \mb C[\bl z], f\in\m H\}}.
\eeqa
Thus
$$
\m H\otimes_{\mathbb C[\bl{z}]}\C_{\bl w} := (\m H \otimes \C) / \mathcal N,
$$
where the module action is defined by the compression of the operator $T_p\otimes 1_{\bl w},\,\, p\in \mathbb C[\bl{z}],$ to the subspace  $(\m H \otimes \C) / \mathcal N.$
We recall the notion of local freeness of a Hilbert module in accordance with \cite[Definition 1.4]{CD}.

\begin{defn}[Definition 1.4, \cite{CD}]\label{locally-free}
Let $\mathcal H$ be a Hilbert module over $\mathbb C[\bl{z}].$
Let $\Omega$ be a bounded open connected subset of $\mathbb C^n.$
We say $\mathcal H$ is locally free of rank $k$ at $\bl w$ in $\Omega^*:=\{\bl z \in \mathbb C^n: \bar{\bl z} \in \Omega\}$ if  there exists a neighbourhood $\Omega_0^*$ of $\bl w$ and holomorphic functions $\gamma_1, \gamma_2, \ldots , \gamma_k:\Omega^*_0 \to \mathcal H$ such that the linear span of the set of $k$ vectors $\{\gamma_1(\bl z), \ldots, \gamma_k(\bl z)\}$  is the module tensor product $\mathcal H \otimes_{\mathbb C[\bl{z}]} \mathbb C_{\bl{z}}.$ Following the terminology of \cite{CD}, we say that a module $\mathcal H$ is {\textit locally free on $\Omega$} of rank $k$  if it is locally free of rank $k$ at every $\bl w$ in $\Omega^*.$
\end{defn}

Let $\mb D^n=\{\bl z: |z_1|, \ldots , |z_n| < 1\}$ be the polydisc in $\mb C^n.$ For $\l>0,$ it is well known that the function $K^{(\l)}:\mb D^n\times\mb D^n\to\mb{C}$ given by the formula
\Bea
K^{(\l)}(\bl z, \bl w)=\prod_{j=1}^n(1-z_j\bar w_j)^{-\l},\,\, \bl z,\bl w\in\mb D^n,
\Eea
is positive definite. The function $K^{(\l)}$ uniquely determines a Hilbert space, say $\mb A^{(\l)}(\mb D^n),$  consisting of holomorphic functions defined on $\mb D^n$ with the reproducing property
$$
\langle f(\cdot) , K^{(\l)}(\cdot,\bl w) \rangle = f(\bl w),\,\, f\in \mb A^{(\l)}(\mb D^n),\,\, \bl w\in \mb D^n.
$$
For $\l>1,$ this coincides with the usual notion of the weighted Bergman spaces   $\mb A^{(\l)}(\mb D^n)$ defined as the Hilbert space of square integrable holomorphic functions on $\mb D^n$ with respect to the measure
$ dV^{(\l)}=\big(\frac{\l-1}{\pi}\big)^n\Big(\prod_{i=1}^n(1-r_i^2)^{\l-2}r_idr_id\theta_i\Big).$ The limiting case of $\l=1$ is the Hardy space $H^2(\mb D^n).$ Throughout the rest of this paper, we will assume that $\l >0.$

The natural action of the permutation group $\mathfrak S_n$ on $\mb C^n.$ is given by the formula:
$$(\sigma,\bl z)\mapsto\sigma\cdot\bl z:=(z_{\sigma^{-1}(1)},\ldots,z_{\sigma^{-1}(n)}),\,\,(\sigma,\bl z)\in \mathfrak S_n\times\mb C^n. $$
The induced action on the Hilbert space $\mb A^{(\l)}(\mb D^n)$ is $f\mapsto f\circ\sigma^{-1},\, \sigma \in \mathfrak S_n.$  Let $\bl s:\mb C^n\to\mb C^n$ be the symmetrization map $\bl s=(s_1,\ldots, s_n),$ where $s_k(\bl z) = \sum_{1\leq i_1, \ldots , i_k\leq n} z_{i_1} \cdots z_{i_k},$ $1 \leq k \leq n.$  Let $(M_1, \ldots , M_n)$ denote the $n$-tuple of multiplication by the coordinate functions $z_i,$ $1\leq i \leq k$ on $\mb A^{(\l)}(\mb D^n).$  Clearly,  $(M_{s_1},\ldots, M_{s_n})$ defines a commuting tuple of bounded linear operators  on $\mb A^{(\l)}(\mb D^n).$ Define $\Delta(\bl z)=\prod_{i<j}(z_i-z_j),$  for $\bl z\in \mb C^n$. Let
 $$
 \m Z=\{\bl z\in \mb D^n\mid \Delta(\bl z)=0\} = \{\bl z\in\mb D^n\mid z_i=z_j \mbox{~for some~} i\neq j, 1\leq i,j\leq n\}
 $$
 and $\mb G_n=\bl s(\mb D^n).$ For every $\bl u\in\mb G_n\setminus \bl s(\m Z),$ we note that the set $\bl s^{-1}(\{\bl u\})$ has exactly $n!$ elements.
If $M_\phi$ is a multiplication operator on $\mb A^{(\l)}(\mb D^n)$ by a holomorphic function $\phi$, then $M_\phi^*K^{(\l)}_{\bl w} = \ov{\phi(\bl w)}K^{(\l)}_{\bl w}$ for $\bl w\in\mb D^n$. Therefore we have the following lemma.

\begin{lem}
For $\sigma\in\mathfrak S_n,$ $i=1,\ldots, n,$ $M_i^*K^{(\l)}_{\bl w_\sigma}=\bar w_{\sigma^{-1}(i)}K^{(\l)}_{\bl w_\sigma} \mbox{~and~} M_{s_i}^*K^{(\l)}_{\bl w_\sigma}=\ov {s_i(\bl w)}K^{(\l)}_{\bl w_\sigma}.$
\end{lem}

Let $\mb C[\bl z]^{\mathfrak S_n}$ be the ring of invariants under the action of $\mathfrak S_n$ on $\mathbb C[\bl z],$ that is, $$\mb C[\bl z]^{\mathfrak S_n} = \{f\in \mathbb C[\bl z]: f(\sigma\cdot \bl z) = f(\bl z), \sigma \in \mathfrak S_n \}.$$  Furthermore,  $\mb C[\bl z]^{\mathfrak S_n}= \mathbb C[s_1, \ldots , s_n],$ see \cite[p. 39]{M}. We now state the main Theorem of this Section.
\begin{thm}\label{free}
The Hilbert module $\mb A^{(\l)}(\mb D^n)$ over $\mb C[\bl z]^{\mathfrak S_n}$ is locally free of rank $n!$ on  $\mb G_n\setminus \bl s(\m Z).$
\end{thm}

The proof is facilitated by breaking it up into several pieces. Some of these pieces make essential use of the fact that $\C[\bl z]$ is a finitely generated free module over $\C[\bl z]^{{\mathfrak S}_n}$ of rank $n!$ \cite[Theorem 1]{Y}. The motivation for the following lemma and some of the subsequent comments come from \cite{Ch}.

\begin{lem}\label{galois}
For any basis $\{p_\sigma\}_{\sigma\in\mathfrak S_n}$ of  $\mb C[\bl z]$ over $\mb C[\bl z]^{\mathfrak S_n},$ we have $$\det \big(\!\!\big(p_{\sigma}(\bl w_\tau)\big)\!\!\big)_{\sigma,\tau\in\mathfrak S_n}\not \equiv 0.$$
\end{lem}
\begin{proof}
 Let $L=\mb C(\bl z)$ denote the field of rational functions and $K=\mb C(\bl z)^{\mathfrak S_n}$ be the field of symmetric rational function. From \cite[Example 2.22]{M}, it is known that $L$ over $K$ is a finite Galois extension with Galois group $\mr{Gal}(L/K)=\mathfrak S_n.$ Let $f\in L$, that is, $f = \frac{p}{q}$ for some polynomials $p$ and $q$. Pick $\tilde {q} = \prod_{\sigma\in\mathfrak S_n}q(\bl z_\sigma)$ and $\tilde p = p\prod_{\sigma\in\mathfrak S_n, \sigma\neq 1}q(\bl z_\sigma).$ Now, $f = \frac{\tilde p}{\tilde q},$ where $\tilde{q}$ is symmetric. Again, since $\{p_\sigma\}_{\sigma\in\mathfrak S_n}$ is a basis for $\mb C[\bl z]$ over the ring $\mb C[\bl z]^{\mathfrak S_n}$, we have $p = \sum_{\sigma\in\mathfrak S_n}p_\sigma h_\sigma$ where $h_\sigma$'s are symmetric polynomial which in turn shows that $f = \sum_{\sigma\in\mathfrak S_n}p_\sigma \frac{h_\sigma}{\tilde q}$. Thus $\{p_\sigma\}_{\sigma\in\mathfrak S_n}$ forms a basis of $L$ over $K$. Now we make use of the following basic result from Galois theory \cite[Lemma 3.4]{C}:

\emph{If $N/F$ is a finite Galois extension with $\mr{Gal}(N/F)=\{g_1,\ldots,g_m\}$ and  $\{e_1,\ldots,e_m\}$ is a $F$-basis of $N$, then $\big(g_1(e_j),\ldots,g_m(e_j)\big)_{j=1}^m$ forms a basis of $F^m/F$.}

Consequently, $\big((p_\sigma\circ\tau^{-1})_{\sigma\in\mathfrak S_n}\big)_{\tau\in\mathfrak S_n}$ is a basis of $L^{n!}/L.$ Hence we have the desired result.
\end{proof}

Recall that the length  of permutation $\sigma\in\mathfrak S_n$ is the number of inversions in $\sigma$ \cite[p.  4]{K}. Here, by an inversion in $\sigma$, we mean a pair $(i, j)$ with $1\leq i < j\leq n$ such that $\sigma(i) > \sigma(j)$. This is the smallest number of transpositions of the form $(i, i+1)$ required to write $\sigma$ as a product of these transpositions.
\begin{lem}\label{det}
Pick a  basis for $\mb C[\bl z]$ over $\mb C[\bl z]^{\mathfrak S_n}$ consisting of homogeneous polynomials $p_\sigma,$ $\sigma\in\mathfrak S_n,$ $\deg p_\sigma=\ell(\sigma).$ Then
\begin{enumerate}
\item[(i)] the determinant $\det\big(\!\!\big(p_{\sigma}(\bl w_\tau)\big)\!\!\big)_{\sigma,\tau\in\mathfrak S_n}$ is a homogeneous polynomial of degree $\frac{n!}{2}\binom{n}{2},$
\item[(ii)]$\det\big(\!\!\big(p_{\sigma}(\bl w_\tau)\big)\!\!\big)_{\sigma,\tau\in\mathfrak S_n}$ is a non-zero constant multiple of $\Delta(\bl w)^{\frac{n!}{2}}.$
\end{enumerate}
\end{lem}
\begin{proof}
Clearly,
\Bea
\det\big(\!\!\big(p_{\sigma}(\bl w_\tau)\big)\!\!\big)_{\sigma,\tau\in\mathfrak S_n}=\sum_{\nu\in\mathfrak S_{n!}}\prod_{\sigma\in\mathfrak S_n}p_{\sigma}(\bl w_{\nu\sigma}).
\Eea
We note that
\Bea
\deg \prod_{\sigma\in\mathfrak S_n}p_{\sigma}(\bl w_{\nu\sigma})=\sum_{\sigma\in\mathfrak S_n}\deg p_{\sigma}(\bl w)=\sum_{\sigma\in\mathfrak S_n}\deg p_{\sigma}=\sum_{\sigma\in\mathfrak S_n}\ell(\sigma).
\Eea
Let $I_n(k)$ denote the number of $k$-inversions in $\mathfrak S_n$ \cite[p.  1]{Ma}. Alternatively, $I_n(k)=\mr{card}\{\sigma\in\mathfrak S_n\mid \ell(\sigma)=k\}.$
Note that
\Bea
\sum_{\sigma\in\mathfrak S_n}\ell(\sigma)=\sum_{k=1}^{\binom{n}{2}}\sum_{\ell(\sigma)=k}\ell(\sigma)=
\sum_{k=1}^{\binom{n}{2}}kI_n(k).
\Eea
The generating function formula for $I_n(k)$ is given by \cite[Theorem 1]{Ma}
\Bea
\sum_{k=1}^{\binom{n}{2}}I_n(k)z^k=\prod_{i=1}^{n-1}\sum_{j=0}^iz^j.
\Eea
Differentiating with respect to $z,$ we obtain
\Bea
\sum_{k=1}^{\binom{n}{2}}kI_n(k)z^{k-1}=\sum_{i=1}^{n-1}(1+\ldots+iz^{i-1})\prod_{j=1,j\neq i}^{r-1}(1+\ldots+z^j).
\Eea
Putting $z=1,$ we have
\Bea
\sum_{k=1}^{\binom{n}{2}}kI_n(k)=\sum_{i=1}^{n-1}\frac{i(i+1)}{2} \prod_{j=1, j\neq i}^{n - 1}(j+1) = \frac{n!}{2}\sum_{i=1}^{n-1}i=\frac{n!}{2}\binom{n}{2}.
\Eea
This proves part (i). For part (ii), let us choose $i, j$ with $1\leq i< j\leq n$. Consider the automorphism of $\mathfrak S_n$ given by $\tau\mapsto \tau (i, j)$, where $(i, j)$ is the transposition. This automorphism  maps an even permutation to an  odd permutation and vice versa. For any polynomial $p,$ clearly, $p (\bl z_{\tau}) = \sum_{m,n} a_{mn}({\bl z}^\i) z_i^mz_j^n\in\C[\bl z],$ where each $a_{mn}({\bl z}^\i)$ is a polynomial in the variables $z_1, \ldots, z_{i -1},z_{i + 1}, \ldots, z_{j -1},z_{j + 1}, \ldots, z_n$. Thus $p(\bl w_{\tau}) - p(\bl w_{\tau (i,j)}) = \sum_{m,n} a_{mn}({\bl w}^\i) (w_i^mw_j^n - w_j^mw_i^n)$  is divisible by $w_i - w_j$. Thus  for each even permutation $\tau$, if we subtract the $\tau(i,j)$-th column $\big(p_\sigma(\bl w_{\tau(i,j)}) \big)_{\sigma\in\mathfrak S_n}$ from $\tau$-th column $(p_\sigma(\bl w_{\tau}) )_{\sigma\in\mathfrak S_n}$, the determinant does not change.   Consequently, we see that $w_i - w_j$ is a factor of the determinant. Since we have exactly $\frac{n!}{2}$  even permutations in $\mathfrak S_n,$ it follows that $(w_i - w_j)^\frac{n!}{2}$ must divide the determinant. This is true for every pair of $i<j$ and $\C[\bl z]$ is a unique factorization domain.  Hence $\Delta(\bl w)^{\frac{n!}{2}}$ divides the determinant. From part (i) and Lemma \ref{galois}, we see that the degree of the polynomial $\Delta(\bl w)^{\frac{n!}{2}}$ is equal to $\frac{n!}{2}\binom{n}{2}$ completing the proof of  part (ii).
\end{proof}
\begin{rem}
The degree of the polynomials in a basis consisting of the Descent polynomials \cite[p.  6]{A} or the Schubert polynomials \cite[Theorem 2.16]{K}, meet the hypothesis made in Lemma \ref{det}.
\end{rem}
\begin{lem} \label{lem2.8}Let $\m H$ be a Hilbert module over $\mb C[\bl z]$ consisting of holomorphic functions defined on the polydisc $\mb D^n$ possessing a reproducing kernel, say $K.$ Assume that $\mb C[\bl z]$ is dense in $\mathcal H.$  If $v$ is in $\cap_{i=1}^n \ker \big(M_{s_i}-s_i(\bl w)\big)^*, \bl w\in\mb D^n\setminus\m Z,$ then there exists unique tuple $(c_\sigma)_{\sigma \in\mathfrak S_n},$ such that $v=\sum c_\sigma K(\cdot,\bl w_{\sigma}).$
\end{lem}
\begin{proof} Clearly, $M_{s_i}^* K(\cdot, w_\sigma) = \overline{s_i(w_\sigma)} K(\cdot, w_\sigma) = \overline{s_i(w)} K(\cdot, w_\sigma).$
To complete the proof, given a joint eigenvector $v,$ it is enough to ensure the existence of a unique tuple $(c_\sigma)_{\sigma\in \mathfrak S_n}$ of complex numbers such that
$$
\langle v, p\rangle = \langle \sum_{\sigma\in\mathfrak S_n} c_\sigma K(\cdot,\bl w_{\sigma}), p\rangle = \sum_{\sigma\in\mathfrak S_n}  c_\sigma \ov {p(\bl w_\sigma)},
$$ for all polynomials $p$ since $\mathbb C[\bl z]$ is dense in the  Hilbert module $\mathcal H$.
In particular, if there exists a unique solution  for some choice of a basis, say $\{p_\tau\}_{\tau\in\mathfrak S_n}$, of $\mb C[\bl z]$ over the ring $\mb C[\bl z]^{\mathfrak S_n}$, then for any $p = \sum_{\tau\in\mathfrak S_n}p_\tau h_\tau \in\C[\bl z]$, we have
\beqa
\langle v, p\rangle &=& \langle v, \sum_{\tau\in\mathfrak S_n}p_\tau h_\tau\rangle = \sum_{\tau\in\mathfrak S_n}\langle M_{h_\tau}^*v, p_\tau\rangle = \sum_{\tau\in\mathfrak S_n}\ov{h_\tau(\bl w)}\langle v, p_\tau\rangle \\ &=& \sum_{\tau\in\mathfrak S_n}\ov{h_\tau(\bl w)}\sum_{\sigma\in\mathfrak S_n}  c_\sigma \ov {p_\tau(\bl w_\sigma)} = \sum_{\sigma\in\mathfrak S_n} c_\sigma  \sum_{\tau\in\mathfrak S_n}\ov{h_\tau(\bl w_\sigma)}\ov {p_\tau(\bl w_\sigma)} \\&=& \sum_{\sigma\in\mathfrak S_n}  c_\sigma \ov {p(\bl w_\sigma)}.
\eeqa
Thus choosing  $\{p_\tau\}_{\tau\in\mathfrak S_n}$ as in the hypothesis of Lemma \ref{det} and using part (ii) of that Lemma, we have a unique solution $(c_\sigma)_{\sigma \in \mathfrak S_n}$ for the system of equations $$\inner{v}{p_\tau}=\sum_{\sigma\in\mathfrak S_n}c_\sigma \ov{p_\tau(\bl w_\sigma)}$$
as long as $\bl w$ is from $\mb D^n\setminus\m Z. $
\end{proof}
As a consequence of the Lemma we have just proved, we see  that the set of vectors $\{K_{\bl w_\sigma}\mid \sigma\in\mathfrak S_n\}$ are both  linearly independent and spanning for the joint kernel $\cap_{i=1}^n \ker \big(M_{s_i}-s_i(\bl w)\big)^*, \bl w\in\mb D^n\setminus\m Z.$ Therefore, we have the following Corollary.

\begin{cor}\label{dim}Let $\m H$ be a Hilbert module over $\mb C[\bl z]$ consisting of holomorphic functions defined on the polydisc $\mb D^n$ possessing a reproducing kernel, say $K.$ Assume that $\mb C[\bl z]$ is dense in $\mathcal H.$ Then $\dim \cap_{i=1}^n \ker \big(M_{s_i}-s_i(\bl w)\big)^*= n!.$
\end{cor}

To complete the proof of Theorem \ref{free}, we need to relate the joint kernel $\cap_{i=1}^n \ker \big(M_{s_i}-s_i(\bl w)\big)^*$ to the module tensor product $\m H\otimes_{\mb C[\bl z]^{\mathfrak S_n}}\mb C_{\bl w}.$ The following Lemma gives an isomorphism between these two.
A special case of \cite[Lemma 5.11]{DP}, included in the Lemma below, is 
used in proving a  generalization of Theorem \ref{free} to $\mb P_{\bl p}\big(\mb A^{(\l)}(\mb D^n) \big).$

\begin{lem}\label{dp}
If $\m H$ is a Hilbert module over $\mb C[\bl z]$ consisting of holomorphic functions defined on some bounded domain $\Omega \subseteq \mb C^n,$ then we have
\begin{enumerate}
\item $\m H \otimes_{\mb C[\bl z]}\mb C_{\bl w}\cong\cap_{p\in\mb C[\bl z]}\ker M^*_{p - p(\bl w)};$
\item $\m H\otimes_{\mb C[\bl z]^{\mathfrak S_n}}\mb C_{\bl w}\cong\cap_{i=1}^n \ker \big(M_{s_i}-s_i(\bl w)\big)^*;$
\item $p_1\otimes_{\mb C[\bl z]^{\mathfrak S_n}}1_{\bl w},\ldots, p_t\otimes_{\mb C[\bl z]^{\mathfrak S_n}}1_{\bl w}$ spans $\m H\otimes_{\mb C[\bl z]^{\mathfrak S_n}}\mb C_{\bl w},$ for any set of generators $p_1,\ldots,p_t$ for $\m H$ over $\mb C[\bl z]^{\mathfrak S_n}.$

\end{enumerate}
\end{lem}
\begin{proof}
We have to show that
$\m H\otimes_{\mb C[\bl z]}\mb C_{\bl w} = \cap_{p\in\mb C[\bl z]}\ker M^*_{p - p(\bl w)}.$
Recall that $\m H\otimes_{\mb C[\bl z]}\mb C_{\bl w}$  is the ortho-complement of the subspace $\m N = \{(p - p(\bl w))f : p\in\mb C[\bl z], f\in\m H\}$ in $\m H\otimes \mb C.$ Therefore, we have
$$
g\in \m N^\perp \iff \inner{g}{\big(p - p(\bl w)\big)f} = 0 \mbox{~for~all~}  p\in\mb C[\bl z], f\in\m H \iff M^*_{\big(p - p(\bl w)\big)}g = 0, \, p\in\mb C[\bl z].
$$

Similarly, $\cap_{p\in\C[\bl z]^{\mathfrak S_n}}\ker M^*_{p - p(\bl w)} \subseteq \cap_{i=1}^n \ker \big(M_{s_i}-s_i(\bl w)\big)^*$. Also, if $f\in \cap_{i=1}^n \ker \big(M_{s_i}-s_i(\bl w)\big)^*$, then $M_{s_i}^*f = \ov{s_i(\bl w)}f,\, 1\leq i\leq n.$ Since $p - p(\bl w)$ is a symmetric polynomial, the existence of a polynomial $q$ such that $p - p(\bl w) = q\circ\bl s$ follows. Thus
$$
M^*_{q\circ\bl s}f = {q(M_{s_1},\ldots, M_{s_n})}^*f = \ov{q(\bl s(\bl w))}f = 0.
$$

To prove the last statement, consider the map $Q: \m H\ra \m H\otimes_{\mb C[\bl z]^{\mathfrak S_n}}\C_{\bl w}$ defined by $Qf = f\otimes_{\mb C[\bl z]^{\mathfrak S_n}}1_{\bl w}$. Note that $Q$ is  the composition of a unitary map from $\m H$ to $\m H\otimes\C$ followed by the quotient map, hence it is onto and $\|Q\|\leq 1.$ Since $p_1\mb C[\bl z]^{\mathfrak S_n}+\cdots+p_t\mb C[\bl z]^{\mathfrak S_n}$ is dense in $\m H$, it follows that  $Q(p_1\mb C[\bl z]^{\mathfrak S_n}+\cdots+p_t\mb C[\bl z]^{\mathfrak S_n})$ is dense in $\m H\otimes_{\mb C[\bl z]^{\mathfrak S_n}}\C_{\bl w}$. Now for any
$\sum_{i=1}^tp_if_i\in\m H$, where $f_i$'s are in $\mb C[\bl z]^{\mathfrak S_n}$, we have
$$
Q\big(\sum_{i=1}^tp_if_i\big) = \big(\sum_{i=1}^tp_if_i\big)\otimes_{\mb C[\bl z]^{\mathfrak S_n}}1_{\bl w} = \sum_{i=1}^tp_i\otimes_{\mb C[\bl z]^{\mathfrak S_n}}f_i\cdot1_{\bl w} = \sum_{i=1}^tf_i(\bl w)p_i\otimes_{\mb C[\bl z]^{\mathfrak S_n}}1_{\bl w}.
$$
Therefore, $Q(p_1\mb C[\bl z]^{\mathfrak S_n}+\cdots+p_t\mb C[\bl z]^{\mathfrak S_n})$ is finite dimensional and hence $\m H\otimes_{\mb C[\bl z]^{\mathfrak S_n}}\C_{\bl w}$ is finite dimensional and is spanned by $p_1\otimes_{\mb C[\bl z]^{\mathfrak S_n}}1_{\bl w},\ldots, p_t\otimes_{\mb C[\bl z]^{\mathfrak S_n}}1_{\bl w}$. \end{proof}

\begin{proof}[Proof of Theorem \ref{free}]
Using Corollary \ref{dim}, we show that the  map $\mathpzc t:\bl u\mapsto \mr{span}\{K^{(\l)}_{\bl w}\mid \bl w\in\bl s^{-1}(\bl u)\}$ taking values in the Grassmannian ${\rm Gr}\big(n!,\mb A^{(\l)}(\mb D^n)\big)$ of the Hilbert space $\mathbb A^{(\l)}(\mathbb D^n)$ of rank $n!$ is anti-holomorphic.
Given any $\bl u_0,$ fixed but arbitrary, in $\mb G_n\setminus \bl s(\m Z),$ there exists a neighborhood of $\bl u_0,$ say $U,$ on which $\bl s$ admits $n!$ local inverses. Enumerate them as $\v_1,\ldots,\v_{n!}.$ Then the linearly independent set
$$\big \{\g_i: \g_{i}(\bl u) = K^{(\l)}\big(\cdot, \v_i(\bl u)\big), u\in U\big \}_{i = 1}^{n!}$$ of anti-holomorphic $\mb A^{(\l)}(\mb D^n)$-valued functions spans the joint kernel $\cap_{i=1}^n \ker \big(M_{s_i}-s_i(\bl w)\big)^*.$
\end{proof}


\begin{rem} \label{sharp} We give a realization of $\mb A^{(\l)}(\mb D^n)$ as a space of holomorphic functions on an open subset of $\big(\mathbb G_n\setminus\bl s(\mathcal Z)\big)^*$ possessing a sharp (reproducing) kernel \cite[Definition 2.1]{AS},
we mimic here the construction in \cite[Remark 2.6]{CS}. Define $\Gamma:\mb A^{(\l)}(\mb D^n)\to\m O(U^*)$ by $\big(\Gamma(f)(\bar{\bl v})\big)_i=\inner{f}{\g_i(\bar{\bl v})}.$ Let $\m H=\Gamma\big(\mb A^{(\l)}(\mb D^n)\big)\subset\m O(U^*).$ Let $\langle \Gamma f, \Gamma g\rangle = \langle f , g \rangle,$ $f,g \in  \mb A^{(\l)}(\mb D^n).$ Equipped with this inner product, $\m H$ is a  Hilbert space. Now, by definition $\Gamma$ is a unitary from $\mb A^{(\l)}(\mb D^n)$ to $\m H.$ Define $K_\Gamma:U^*\times U^*\to M_{n!}(\mb C)$ by $\inner{K_\Gamma(\bar{\bl u},\bar{\bl v})e_j}{e_i}=\inner{\g_j(\bar{\bl v})}{\g_i(\bar{\bl u})} = \langle \Gamma\big(\g_j(\bar {\bl v})\big)(\bar{\bl u}), e_i\rangle$, that is, $K_\Gamma(\cdot,\bar{\bl v})e_j = \Gamma\big(\g_j(\bar {\bl v})\big)$. The string of equalities
\Bea
\inner{\Gamma f}{K_\Gamma(\cdot, \bar{\bl v})e_j} =  \inner{\Gamma f}{\Gamma\big(\g_j(\bar {\bl v})\big)} = \inner{f}{\g_j(\bar {\bl v})} = \inner{\Gamma(f)(\bar{\bl v})}{e_j}
\Eea
shows that $\m H$ is a reproducing kernel Hilbert space with $K_\Gamma$ as reproducing kernel. Note that
\beqa
\inner{M_i^*K_\Gamma(\cdot, \bar{\bl v})e_j}{\Gamma f} = \inner{K_\Gamma(\cdot, \bar{\bl v})e_j}{\bl u_i\Gamma f} = \ov{\inner{\bar{\bl v_i}\Gamma f(\bl v)}{e_j}} = \inner{\bl v_iK_\Gamma(\cdot, \bar{\bl v})e_j}{\Gamma f},
\eeqa
that is, $M_i^*K_\Gamma(\cdot, \bar{\bl v})e_j = \bl v_iK_\Gamma(\cdot, \bar{\bl v})e_j.$ Thus
\beqa
\Gamma M_{s_i}^* K^{(\l)}\big(\cdot, \phi_j(\bar{\bl v})\big) &=& \Gamma \{\ov{s_i\big(\phi_j(\bar{\bl v})\big)}K\big(\cdot, \phi_j(\bar{\bl v})\big)\} = \bl v_i\Gamma K\big(\cdot, \phi_j(\bar{\bl v})\big)  = \bl v_i\Gamma\big(\g_j(\bar {\bl v})\big)\\&=& \bl v_iK_\Gamma(\cdot, \bar{\bl v})e_j = M_i^*K_\Gamma(\cdot, \bar{\bl v})e_j \\&=& M_i^*\Gamma K^{(\l)}\big(\cdot, \phi_j(\bar{\bl v})\big).
\eeqa
Since the linear span $K(\cdot, \bl w),$ $\bl w \in U,$ where $U \subseteq \mb D^n$ is any small open set, is dense in $\mb A^{(\l)}(\mb D^n),$ it follows that $\Gamma M_{s_i}^*=M_i^*\Gamma.$ Consequently, $\Gamma$ is a module isomorphism between $\mb A^{(\l)}(\mb D^n)$ and $\m H.$ So $\Gamma$ is a unitary map from $\cap_{i=1}^n \ker \big(M_{s_i}-s_i(\bl w)\big)^*$ to $\cap_{i=1}^n \ker (M_i - \bl u_i)^*$, where $\bl s(\bl w) = \bl u.$ This  shows that ran$K_\Gamma(\cdot, \bl u)e_j = \cap_{i=1}^n \ker (M_i - \bl u_i)^*$, that is, $K_\Gamma$ is sharp.
\end{rem}

We would now make use of the following well known result, which is analogous to the statement: The polynomial ring $\C[\bl z]$ is a finitely generated free module over $\C[\bl z]^{{\mathfrak S}_n}$ of rank $n!$.
\begin{thm}\label{C} The module $\mb P_{\bl p}\C[\bl z]$ is a finitely generated free module over $\C[\bl z]^{{\mathfrak S}_n}$ of rank $\chi_{\bl p}(1)^2$.
\end{thm}
We are unable to locate a proof of this Theorem and therefore indicate a proof
using results from \cite{S}.
\begin{proof} Set $\C[\bl z]_{\bl p} :=  \mb P_{\bl p}\C[\bl z].$  There exists a set of homogeneous polynomials in  $\C[\bl z]_{\bl p},$ whose images in the quotient module $\m S_{\bl p} = \C[\bl z]_{\bl p}/ \{s_1\C[\bl z]_{\bl p} + \cdots+s_n\C[\bl z]_{\bl p}\}$ forms a $\C$-basis for $\m S_{\bl p},$ see \cite[Theorem 1.3]{S}).  Also, from \cite[Theorem 3.10]{S}, it follows that  $p_1,\ldots,p_\mu$ is a free basis for $\C[\bl z]_{\bl p}$ over
$\C[\bl z]^{{\mathfrak S}_n}$. Now to see that $\mu  = \chi_{\bl p}(1)^2$, we make use of \cite[Theorem 4.9]{S} and its proof along with \cite[Corollary 4.9]{S}. It says that the action of $\mathfrak S_n$ on the quotient ring $\C[\bl z]/ \{s_1\C[\bl z] + \cdots+s_n\C[\bl z]\} \cong \oplus_{\bl p\vdash n }\m S_{\bl p}$ is isomorphic to the regular representation of $\mathfrak S_n,$ where the action on $\m S_{\bl p}$ is isomorphic to the representation $\pi_{\bl p}$ corresponding to $\bl p\vdash n$ with multiplicity $\chi_{\bl p}(1).$\end{proof}


Recall that the rank of a Hilbert module $\mathcal H$ over a ring $\mathcal R$  is $\inf  |\mathcal F|,$ where $\mathcal F \subseteq \mathcal H$ is any subset with the property $\{r_1 f_1 + \cdots + r_k f_k: f_1, \ldots , f_k \in \mathcal H; r_1, \ldots , r_k \in \mathcal R\}$ is dense in $\mathcal H$ and $|\mathcal F|$ denotes the cardinality of $\mathcal F$ (cf. \cite[Section 2.3]{CG}).
The proof of the following Corollary is immediate from Theorem \ref{C} and Lemma \ref{dp}.

\begin{cor}\label{fg}
The Hilbert module $\mb P_{\bl p}\big(\mb A^{(\l)}(\mb D^n)\big)$ over $\mb C[\bl z]^{\mathfrak S_n}$ is of rank at most  $\chi_{\bl p}(1)^2.$
\end{cor}

Let $M_{s_i}^{(\bl p)}=M_{s_i}{\mid_{\mb P_{\bl p}\big(\mb A^{(\l)}(\mb D^n)\big)}}.$ From \cite[Remark 3.9]{BS}, recall that each $\mb P_{\bl p}\big(\mb A^{(\l)}(\mb D^n)\big)$ is a reducing subspace of $M_{s_i}$ for each $i, 1\leq i\leq n.$ Therefore, $M_{s_i}^*=\oplus_{\bl p\vdash n}\big(M^{(p)}_{s_i}\big)^*$ and we have
$$\cap_{i=1}^{n}\ker M^*_{s_i-s_i(\bl w)}=\oplus_{\bl p\vdash n }\cap_{i=1}^n\ker\big(M^{(\bl p)}_{s_i-s_i(\bl w)}\big)^*.$$
Now we have the following useful Proposition.
\begin{prop}
$\mr{dim}\cap_{i=1}^n\ker\big(M^{(\bl p)}_{s_i-s_i(\bl w)}\big)^*=\chi_{\bl p}(1)^2,$ $w\in \mathbb D^n \setminus \mathcal Z.$
\end{prop}
\begin{proof}
From Corollary \ref{fg} and Lemma \ref{dp}, it follows that $\mr{dim}\cap_{i=1}^n\ker\big(M^{(\bl p)}_{s_i-s_i(\bl w)}\big)^* \leq\chi_{\bl p}(1)^2$. However if it is strictly less for some $\bl p\vdash n$ we have the following contradiction:
$$
n! =\dim \cap_{i=1}^{n}\ker M^*_{s_i-s_i(\bl w)} = \sum_{\bl p\vdash n }\mr{dim}\cap_{i=1}^n\ker\big(M^{(\bl p)}_{s_i-s_i(\bl w)}\big)^* < \sum_{\bl p\vdash n }\chi_{\bl p}(1)^2 = n!.
$$
For the last equality, see \cite[Theorem 3.4]{KS}.
\end{proof}

From the Proposition given above  and the proof of Theorem \ref{free}, the following generalization to $\mb P_{\bl p}\big(\mb A^{(\l)}(\mb D^n)$  is evident.

\begin{cor}\label{cor2.13}
The Hilbert module $\mb P_{\bl p}\big(\mb A^{(\l)}(\mb D^n)\big)$ over $\mb C[\bl z]^{\mathfrak S_n}$ is locally free of rank $\chi_{\bl p}(1)^2$ on  $\mb G_n\setminus \bl s(\m Z).$
\end{cor}

\begin{rem} Since $\mb P_{\bl p}\big(\mb A^{(\l)}(\mb D^n)\big)$ is assumed to be locally free at $\bl w\in\mb G_n\setminus \bl s(\m Z),$ it follows that
$E_{\bl p}=\{(\bl u, x)\in U\times\mb P_{\bl p}\big(\mb A^{(\l)}(\mb D^n)\big)\mid x\in \cap_{i = 1}^n\ker\big(M_{s_i-u_i})^* \}$ and $\pi(\bl u, x)=\bl u$ defines a rank $\chi_{\bl p}(1)^2$ hermitian anti-holomorphic vector bundle on some open neighbourhood  $W$ of $\bl w.$ The equivalence class of this vector bundle $E_{\bl p}$ determines the isomorphism class of the module $\mb P_{\bl p}\big(\mb A^{(\l)}(\mb D^n)\big)$  and conversely. The vector bundle $E$ corresponding to the module $\mb A^{(\l)}(\mb D^n)$ is therefore the direct sum
$\oplus_{\bl p\vdash n}E_{\bl p}.$
\end{rem}

\begin{rem}
An alternative proof of the Corollary \ref{dim} is possible using Lemma \ref{dp}. For this proof, which is indicated below, it is essential to  use a non-trivial result  from \cite{CS} rather than the direct proof that we have presented earlier. From Lemma \ref{dp}, it follows that $\mr {dim} \cap_{i=1}^n \ker \big(M_{s_i}-s_i(\bl w)\big)^*\leq n!$. To prove the reverse inequality, we show that for $\bl w\in \mb D^n\setminus \m Z,$ the set of vectors $\{K_{\bl w_\sigma}\mid \sigma\in\mathfrak S_n\}$ are  linearly independent.  Since the polynomial ring is dense in $\mb A^{(\l)}(\mb D^n)$, the reproducing kernel $K$ is non-degenerate. From \cite[Lemma 3.6]{CS}, it then follows that $K$ is strictly positive, that is,  for all $k\geq 1$ the $k\times k$-operator matrix $\big(\!\!\big(K(\bl z_i,\bl z_j)\big)\!\!\big)_{1\leq i, j\leq k}$ is injective for every collection $\{\bl z_1,\ldots, \bl z_k\}$ of distinct points in $\mb D^n\setminus \m Z$. Since the set $\{\bl w_\sigma\mid \sigma\in\mathfrak S_n\}$ contains exactly $n!$ distinct points for every $\bl w\in \mb D^n\setminus \m Z$, the matrix $\big(\!\!\big(\inner{K_{\bl w_\sigma}}{K_{\bl w_\tau}}\big)\!\!\big)_{\sigma,\tau\in\mathfrak S_n}$ is injective and hence the nonsingularity of the grammian of $\{K_{\bl w_\sigma}\mid \sigma\in\mathfrak S_n\}$ gives the linear independence.
\end{rem}
\section{$\mathfrak S_n$-invariant kernel}
Let $\Omega\subseteq\mb C^n$ be a bounded domain invariant under the action of $\mathfrak S_n.$ Let $K$ be a $\mathfrak S_n$-invariant reproducing kernel on $\Omega,$ that is,
$$
K(\sigma\cdot\bl z, \sigma\cdot\bl w) = K(\bl z, \bl w) \mbox{~for ~all~} \sigma \in \mathfrak S_n.
$$
Let $\mathcal H_K$ denote the Hilbert space with $K$ as reproducing kernel. Let $U:\mathfrak S_n\ra \mathcal B(\mathcal H_K)$ be a unitary representation. Consider a function $f:\mathfrak S_n\ra \C$ satisfying $f(\sigma^{-1}) = \ov{f(\sigma)}$. Define an operator on $\mathcal H_K$ by
\Bea
T^f = \sum_{\sigma\in \mathfrak S_n} \ov{f(\sigma)}U(\sigma).
\Eea
Since $U(\sigma)^* = U(\sigma^{-1}),$ it follows that
$$
(T^f)^* = \sum_{\sigma\in \mathfrak S_n} {f(\sigma)}U(\sigma)^* = \sum_{\sigma\in \mathfrak S_n} \ov{f(\sigma^{-1})}U(\sigma^{-1}) = \sum_{\tau\in \mathfrak S_n} \ov{f(\tau)}U(\tau) = T^f. $$ Thus we have proved:
\begin{lem}\label{sa}
$T^f$ is self adjoint on $\mathcal H_K$.
\end{lem}
As before, let $\bl\pi_{\bl p}$ be a unitary representation of $\mathfrak S_n$ in the equivalence class of $\bl p\vdash n$, that is, $\bl\pi_{\bl p}(\sigma) = \big (\!\big (\bl\pi_{\bl p}^{ij}(\sigma)\big )\!\big )_{i,j = 1}^{m} \in \C^{m\times m}$, $\sigma \in \mathfrak S_n,$ where $m = \chi_{\bl p}(1)$ and $\chi_{\bl p}$ is the character of the representation $\bl\pi_{\bl p}.$
The following orthogonality relations \cite[Proposition 2.9]{KS} play a central role in this section:
\bea\label{or}
\sum_{\sigma\in\mathfrak S_n}\bl\pi_{\bl p}^{ij}(\sigma^{-1})\bl\pi_{\bl q}^{lm}(\sigma) =
\frac{n!}{\chi_{\bl p}(1)}\delta_{\bl p\bl q}\delta_{im}\delta_{jl},
\eea
where $\delta$ is the Kronecker symbol.  Define the operators $\mathbb{P}_{\bl p}^{ij}, \mathbb{P}_{\bl p}: \mathcal H_K\ra\mathcal H_K,\, 1\leq i, j\leq \chi_{\bl p}(1),$ by
\Bea
\mathbb{P}_{\bl p}^{ij} = \frac{\chi_{\bl p}(1)}{n!}\sum_{\sigma\in\mathfrak{S}_n}\bl\pi_{\bl p}^{ji}(\sigma^{-1})U(\sigma)
\Eea
and
\Bea
\mb P_{\bl p}=\frac{\chi_{\bl p}(1)}{n!}\sum_{\sigma\in\mathfrak S_n}\ov{\chi_{\bl p}(\sigma)}U(\sigma).
\Eea
Clearly,
\Bea\label{rel}\sum_{i = 1}^{\chi_{\bl p}(1)} \mathbb{P}_{\bl p}^{ii} = \mathbb{P}_{\bl p}.\Eea
The following lemma and some of the subsequent discussions are adapted from  the properties of projection operators given in \cite[p.  162]{KS}. We include this for sake of completeness.
\begin{prop}\label{proj1}
For $1\leq i, j\leq \chi_{\bl p}(1)$ and $1\leq l,m\leq \chi_{\bl q}(1)$, $\mathbb{P}_{\bl p}^{ij}\mathbb{P}_{\bl q}^{lm} = \delta_{\bl p\bl q}\delta_{jl}\mathbb{P}_{\bl p}^{im}$.
\end{prop}
\begin{proof}Since $\mathbb{P}_{\bl p}^{ij} = \frac{\chi_{\bl p}(1)}{n!}\sum_{\sigma\in\mathfrak{S}_n}\bl\pi_{\bl p}^{ji}(\sigma^{-1})U(\sigma),$ we have that
\Bea
\mathbb{P}_{\bl p}^{ij}\mathbb{P}_{\bl q}^{lm} &=& \frac{\chi_{\bl q}(1)}{n!}\sum_{\sigma\in\mathfrak{S}_n}\bl\pi_{\bl q}^{ml}(\sigma^{-1})\mathbb{P}_{\bl p}^{ij}U(\sigma)\\ &=& \frac{\chi_{\bl p}(1)\chi_{\bl q}(1)}{(n!)^2}\sum_{\sigma\in\mathfrak{S}_n}\bl\pi_{\bl q}^{ml}(\sigma^{-1})\{\sum_{\tau\in\mathfrak{S}_n}\bl\pi_{\bl p}^{ji}(\tau^{-1})U(\tau)\}U(\sigma)\\&=& \frac{\chi_{\bl p}(1)\chi_{\bl q}(1)}{(n!)^2}\sum_{\sigma\in\mathfrak{S}_n}\sum_{\tau\in\mathfrak{S}_n}\bl\pi_{\bl q}^{ml}(\sigma^{-1})\bl\pi_{\bl p}^{ji}(\tau^{-1})U(\tau)U(\sigma).
\Eea
Let $\eta = \tau\sigma$. Then $\tau^{-1} = \sigma\eta^{-1}$ and
$$
\bl\pi_{\bl p}^{ji}(\sigma\eta^{-1}) = (\bl\pi_{\bl p}(\sigma\eta^{-1}) )_{ji} = (\bl\pi_{\bl p}(\sigma)\bl\pi_{\bl p}(\eta^{-1}) )_{ji} = \sum_{k = 1}^{\chi_{\bl p}(1)} \bl\pi_{\bl p}^{jk}(\sigma)\bl\pi_{\bl p}^{ki}(\eta^{-1}).
$$
Thus, we also have
\Bea
\mathbb{P}_{\bl p}^{ij}\mathbb{P}_{\bl q}^{lm}
&=& \frac{\chi_{\bl p}(1)\chi_{\bl q}(1)}{(n!)^2}\sum_{\sigma\in\mathfrak{S}_n}\sum_{\eta\in\mathfrak{S}_n}\bl\pi_{\bl q}^{ml}(\sigma^{-1})\bl\pi_{\bl p}^{ji}(\sigma\eta^{-1})U(\eta)\\&=& \frac{\chi_{\bl p}(1)\chi_{\bl q}(1)}{(n!)^2}\sum_{\sigma\in\mathfrak{S}_n}\sum_{\eta\in\mathfrak{S}_n}\bl\pi_{\bl q}^{ml}(\sigma^{-1})\sum_{k = 1}^{\chi_{\bl p}(1)} \bl\pi_{\bl p}^{jk}(\sigma)\bl\pi_{\bl p}^{ki}(\eta^{-1})U(\eta)\\&=& \frac{\chi_{\bl p}(1)\chi_{\bl q}(1)}{(n!)^2} \sum_{\eta\in\mathfrak{S}_n}\sum_{k = 1}^{\chi_{\bl p}(1)} \{\sum_{\sigma\in\mathfrak{S}_n}\bl\pi_{\bl q}^{ml}(\sigma^{-1})\bl\pi_{\bl p}^{jk}(\sigma)\}\bl\pi_{\bl p}^{ki}(\eta^{-1})U(\eta)\\&=& \frac{\chi_{\bl p}(1)\chi_{\bl q}(1)}{(n!)^2} \sum_{\eta\in\mathfrak{S}_n}\sum_{k = 1}^{\chi_{\bl p}(1)} \{\delta_{\bl p\bl q}\delta_{lj}\delta_{mk}\frac{n!}{\chi_{\bl q}(1)}\}\bl\pi_{\bl p}^{ki}(\eta^{-1})U(\eta),\, (\mbox{from~Equation~} \eqref{or})\\&=& \delta_{\bl p\bl q}\delta_{jl}\frac{\chi_{\bl p}(1)}{n!}\sum_{\eta\in\mathfrak{S}_n}\sum_{k = 1}^{\chi_{\bl p}(1)} \delta_{mk}\bl\pi_{\bl p}^{ki}(\eta^{-1})U(\eta)\\&=& \delta_{\bl p\bl q}\delta_{jl}\frac{\chi_{\bl p}(1)}{n!}\sum_{\eta\in\mathfrak{S}_n}\bl\pi_{\bl p}^{mi}(\eta^{-1})U(\eta)\\&=& \delta_{\bl p\bl q}\delta_{jl}\mathbb{P}_{\bl p}^{im}.
\Eea
\end{proof}
\begin{cor}\label{decomp}
For each partition $\bl p$ of $n$ and $1\leq i\leq \chi_{\bl p}(1)$, $\mathbb{P}_{\bl p}^{ii}$ is an orthogonal projection and $\sum_{\bl p\vdash n}\sum_{i = 1}^{\chi_{\bl p}(1)} \mathbb{P}_{\bl p}^{ii} = {\mr {id}}.$
\end{cor}
\begin{proof}
Since $\bl\pi_{\bl p}$ is a unitary representation, it follows that $\bl\pi_{\bl p}^{ii}(\sigma^{-1}) = \ov{\bl\pi_{\bl p}^{ii}(\sigma)}$. Thus from Lemma \ref{sa}, we find that $\mathbb{P}^{ii}_{\bl p}$ is self adjoint. From the Proposition \ref{proj1}, it follows that $(\mathbb{P}^{ii}_{\bl p})^2 = \mathbb{P}^{ii}_{\bl p}$. Then we see that
$$
\sum_{\bl p\vdash n}\sum_{i = 1}^{\chi_{\bl p}(1)} \mathbb{P}^{\bl p}_{ii} = \sum_{\bl p~ \vdash n} \mb P_{\bl p} = \sum_{\bl p~ \vdash n} \frac{\chi_{\bl p}(1)}{n!}\sum_{\sigma\in \mathfrak S_n}\chi_{\bl p}(\sigma)U(\sigma) = \frac{1}{n!} \sum_{\sigma\in \mathfrak S_n}\Big(\sum_{\bl p~ \vdash n}\chi_{\bl p}(1)\chi_{\bl p}(\sigma)\Big) U(\sigma) = {\mr {id}},
$$
where the last equality follows from the orthogonality relations \cite[Proposition 3.8]{KS}. This completes the proof.
\end{proof}

For any  $\mathfrak{S}_n$-invariant kernel $K,$  we claim that the function $f\circ\sigma^{-1},$ $\sigma \in \mathfrak S_n,$ is in $\mathcal{H}_K,$ $f$ in $\mathcal{H}_K.$
To see this, recall that  $f$ is in $\mathcal{H}_K$ if and only if there exists a positive real number $c$ such that $K_f(z,w):=\big (c^2K(\bl z, \bl w) - f(\bl z)\ov{f(\bl w)}\big )
$ is positive definite, see \cite[p.  194]{BM}. Since
\begin{eqnarray*}
K_{f\circ \sigma^{-1}}(z,w) &=& c^2K(\bl z, \bl w) - f\circ\sigma^{-1}(\bl z)\ov{f\circ\sigma^{-1}(\bl w)}\\
&=& c^2K(\sigma\cdot\bl u, \sigma\cdot\bl v) - f(\bl u)\ov{f(\bl v)}\\
&=& c^2K(\bl u, \bl v) - f(\bl u)\ov{f(\bl v)}\\
&=&K_f(\bl u, \bl v),
\end{eqnarray*}
where $\sigma\cdot\bl u = \bl z$ and $\sigma\cdot\bl v = \bl w,$ it follows that $K_{f\circ \sigma^{-1}}$ is positive definite.  Thus the operator $R_\sigma: \mathcal{H}_K\ra\mathcal{H}_K,$ $R_\sigma(f) = f\circ\sigma^{-1},$ is well defined.
\begin{lem}
The map $R: \sigma\mapsto R_\sigma$ is a unitary representation of $\mathfrak{S}_n$ on $\mathcal H_K$.
\end{lem}
\begin{proof}
Note that $R_{\sigma\tau}f (\bl z) = f\circ(\sigma\tau)^{-1}(\bl z) = f(\tau^{-1}\sigma^{-1}\cdot\bl z) = (R_\tau f)(\sigma^{-1}\cdot\bl z) = R_\sigma(R_\tau f)(\bl z)$. Thus $R_{\sigma\tau} = R_\sigma R_\tau$.  Since the set $\{K_{\bl w}\mid \bl w\in\Omega\}$ is total in $\mathcal H_K$, it is enough to check $R_\sigma$ is unitary on $\{K_{\bl w}\mid \bl w\in\Omega\}$. Also,
$$
R_\sigma K_{\bl w}(\bl z) = K_{\bl w}(\sigma^{-1}\cdot\bl z) = K(\sigma^{-1}\cdot\bl z, \bl w) = K(\bl z, \sigma\cdot\bl w) = K_{\sigma\cdot\bl w}(\bl z),
$$
that is, $R_\sigma K_{\bl w} = K_{\sigma\cdot\bl w}$. Thus
$$
\langle R_\sigma K_{\bl w}, R_\sigma K_{\bl w'}\rangle =  \langle K_{\sigma\cdot\bl w}, K_{\sigma\cdot\bl w'}\rangle = K(\sigma\cdot\bl w', \sigma\cdot\bl w) = K(\bl w', \bl w) = \langle K_{\bl w}, K_{\bl w'}\rangle.
$$
This completes the proof.
\end{proof}
For $\l >0,$ recall that $K^{(\l)}:\mb D^n\times\mb D^n\to\mb C$ is the
reproducing kernel of $\mb A^{(\l)}(\mb D^n).$
In the remaining portion of this section, we will specialize to the representation  $R.$
Now, the formula for $\mathbb{P}_{\bl p}^{ij}$ and $\mb P_{\bl p}$ simplifies to
\bea\label{fpf}
\mathbb{P}_{\bl p}^{ij}f(\bl z) = \frac{\chi_{\bl p}(1)}{n!}\sum_{\sigma\in\mathfrak{S}_n}\bl\pi_{\bl p}^{ji}(\sigma^{-1})(R_\sigma f)(\bl z) = \frac{\chi_{\bl p}(1)}{n!}\sum_{\sigma\in\mathfrak{S}_n}\bl\pi_{\bl p}^{ji}(\sigma^{-1})f(\sigma^{-1}\cdot\bl z)
\eea
and
\bea\label{od}
\mb P_{\bl p}f(\bl z) = \frac{\chi_{\bl p}(1)}{n!}\sum_{\sigma\in\mathfrak S_n}\ov{\chi_{\bl p}(\sigma)}R_\sigma f(\bl z) = \frac{\chi_{\bl p}(1)}{n!}\sum_{\sigma\in\mathfrak S_n}\ov{\chi_{\bl p}(\sigma)}f(\sigma^{-1}\cdot\bl z).
\eea
This is the projection formula used extensively earlier in \cite[Equation (3.2)]{BS}. From  Corollary \ref{decomp}, it follows that the subspace $\mathbb{P}_{\bl p}^{ii}\big(\mb A^{(\l)}(\mb D^n)\big)$ of $\mb A^{(\l)}(\mb D^n)$ is a reproducing kernel Hilbert space for each $\bl p\vdash n$ and $1\leq i\leq \chi_{\bl p}(1).$
From Equation \eqref{rel} and Proposition \ref{proj1}, we  have
\bea\label{new}\mathbb{P}_{\bl p}\big(\mb A^{(\l)}(\mb D^n)\big) = \bigoplus_{i = 1}^{\chi_{\bl p}(1)}\mathbb{P}_{\bl p}^{ii}\big(\mb A^{(\l)}(\mb D^n)\big)\eea
and consequently, using Corollary \ref{decomp}, we obtain a finer decomposition of $\mb A^{(\l)}(\mb D^n):$
\bea\label{finedecomp}
\mb A^{(\l)}(\mb D^n) = \bigoplus_{\bl p\vdash n}\mathbb{P}_{\bl p}\big(\mb A^{(\l)}(\mb D^n)\big)= \bigoplus_{\bl p\vdash n}\bigoplus_{i = 1}^{\chi_{\bl p}(1)}\mathbb{P}_{\bl p}^{ii}\big(\mb A^{(\l)}(\mb D^n)\big). \eea
The first of the two equalities was obtained in  \cite[p. 6237 - 6238]{BS}, see also \cite[p.  2368]{MSZ}.
From Lemma \ref{fg}, it follows that the orthogonal projection $\mathbb{P}_{\bl p}$ is non-trivial. In fact, the projections $\mathbb{P}_{\bl p}^{ii}$ are non-trivial as well. We record  this as a separate Lemma. The main ingredient of the proof is borrowed from \cite[p. - 166]{KS}.
\begin{lem}\label{projPii}
For each $\bl p\vdash n$ and $1\leq i\leq \chi_{\bl p}(1)$, $\mathbb{P}_{\bl p}^{ii}\neq 0$.
\end{lem}
\begin{proof}
From Proposition \ref{proj1}, we have
$$
\mathbb{P}_{\bl p}^{ij}\mathbb{P}_{\bl p}^{jj} = \mathbb{P}_{\bl p}^{ij} = \mathbb{P}_{\bl p}^{ii}\mathbb{P}_{\bl p}^{ij},
$$ and it then follows that $\mathbb{P}_{\bl p}^{ij}\mathbb{P}_{\bl p}^{jj}\big(\mb A^{(\l)}(\mb D^n)\big) \subseteq \mathbb{P}_{\bl p}^{ii}\big(\mb A^{(\l)}(\mb D^n)\big)$. Also for $f\in \mb A^{(\l)}(\mb D^n)$,
$$
\mathbb{P}_{\bl p}^{ii}f = \mathbb{P}_{\bl p}^{ij}\mathbb{P}_{\bl p}^{ji}f =  \mathbb{P}_{\bl p}^{ij}\mathbb{P}_{\bl p}^{jj}\mathbb{P}_{\bl p}^{ji}f
$$
and thus $\mathbb{P}_{\bl p}^{ii}\big(\mb A^{(\l)}(\mb D^n)\big)\subseteq \mathbb{P}_{\bl p}^{ij}\mathbb{P}_{\bl p}^{jj}\big(\mb A^{(\l)}(\mb D^n)\big)$. Consequently, $\mathbb{P}_{\bl p}^{ij}$ is a surjective map from $\mathbb{P}_{\bl p}^{jj}\big(\mb A^{(\l)}(\mb D^n)\big)$ onto $\mathbb{P}_{\bl p}^{ii}\big(\mb A^{(\l)}(\mb D^n)\big)$.
Now $\mathbb{P}_{\bl p}^{ij}\mathbb{P}_{\bl p}^{jj}f = 0$ implies that $\mathbb{P}_{\bl p}^{ji}\mathbb{P}_{\bl p}^{ij}\mathbb{P}_{\bl p}^{jj}f = 0$ and hence $\mathbb{P}_{\bl p}^{jj}f = (\mathbb{P}_{\bl p}^{jj})^2f = 0$. This shows that $\mathbb{P}_{\bl p}^{ij}$ is injective on $\mathbb{P}_{\bl p}^{jj}\big(\mb A^{(\l)}(\mb D^n)\big)$. The operator  $\mathbb{P}_{\bl p}^{ij}$, being a finite linear combination of unitaries,  is bounded and hence an invertible map (by Open mapping theorem) from $\mathbb{P}_{\bl p}^{jj}\big(\mb A^{(\l)}(\mb D^n)\big)$ onto $\mathbb{P}_{\bl p}^{ii}\big(\mb A^{(\l)}(\mb D^n)\big)$. Since each $\mathbb{P}_{\bl p}$ is non-trivial, from Equation \eqref{new}, it follows that each $\mathbb{P}_{\bl p}^{ii}$ is non-trivial.
\end{proof}

\begin{prop}\label{reduce}
For each $\bl p\vdash n,\, 1\leq i\leq \chi_{\bl p}(1)$ and $k=1,\ldots,n,\, M_{s_k}\mb P_{\bl p}^{ij} = \mb P_{\bl p}^{ij}M_{s_k}$.
\end{prop}
\begin{proof} For $f\in \mb A^{(\l)}(\mb D^n)$, from the Equation \eqref{fpf} we have
\Bea
\big (M_{s_k}\mb P_{\bl p}^{ij}f \big ) (\bl z) &=& \frac{\chi_{\bl p}(1)}{n!}\sum_{\sigma\in\mathfrak S_n}\bl\pi_{\bl p}^{ji}(\sigma^{-1})M_{s_k}f(\sigma^{-1}\cdot\bl z)\\
&=&\frac{\chi_{\bl p}(1)}{n!}\sum_{\sigma\in \mathfrak S_n}\bl\pi_{\bl p}^{ji}(\sigma^{-1})s_k(\sigma^{-1}\cdot\bl z)f(\sigma^{-1}\cdot\bl z)\\
&=&\frac{\chi_{\bl p}(1)}{n!}\sum_{\sigma\in \mathfrak S_n}\bl\pi_{\bl p}^{ji}(\sigma^{-1})(s_kf)(\sigma^{-1}\cdot\bl z)\\
&=&\big ( \mb P_{\bl p}^{ij} M_{s_k}f \big )(\bl z).
\Eea
This completes the proof.
\end{proof}
In particular for each $\bl p\vdash n$ and $i,\, 1\leq i\leq \chi_{\bl p}(1)$, the projections $\mathbb{P}_{\bl p}^{ii}$ commute with $M_{s_k}$ for each $k, 1\leq k\leq n$ and we have the following corollary.
\begin{cor}\label{fd}
$\mb P_{\bl p}^{ii}\big(\mb A^{(\l)}(\mb D^n)\big)$ is a joint reducing subspace for $M_{s_k},k=1,\ldots,n,$ for every partition $\bl p$ of $n$ and for each $i,\, 1\leq i\leq \chi_{\bl p}(1)$.
\end{cor}

\subsection{$\boldsymbol{\Gamma_n}$ - Contractions} Set $M_{s_k}^{(\bl p, i)} := M_{s_k}{\mid_{\mb P_{\bl p}^{ii}\big(\mb A^{(\l)}(\mb D^n)\big)}},$ $1\leq k \leq n.$ To find the spectrum of the
commuting $n$-tuple $(M_{s_1}^{(\bl p, i)}, \ldots , M_{s_n}^{(\bl p, i)}),$
we first prove, following \cite[Lemma 1.2]{JLT}, a Proposition giving a spectral inclusion for the direct sum of two commuting $n$-tuples.
\begin{prop}\label{Joe}
Let $\boldsymbol S_1$ and $\bl S_2$ be two commuting $n$-tuples of bounded linear operators acting the Hilbert spaces $\m H_1$ and $\m H_2,$ respectively. Then the Taylor joint  spectrum $\sigma({\bl S_1})$ and  $\sigma(\bl S_2)$ are contained in the Taylor joint  spectrum $\sigma(\boldsymbol S_1 \oplus \bl S_2).$
\end{prop}
\begin{proof}
Let $\iota: \mathcal H_1 \oplus \{0\}  \to \mathcal H_1 \oplus \mathcal H_2$ be the inclusion map, $(f,0) \mapsto (f,0)$ and $P:\mathcal H_1 \oplus \mathcal H_2 \to \{0\} \oplus  \mathcal  H_2 $ be the projection, $(f,g) \mapsto (0,g).$
Apply Lemma 1.2  of \cite{JLT} to the short exact sequence  $$0\to\mathcal H_1 \oplus \{0\} \stackrel{\iota}{\to} \mathcal H_1 \oplus \mathcal H_2 \stackrel{P}{\to} \{0\} \oplus  \mathcal  H_2 \to 0$$ and the direct sum
$\bl S_1 \oplus \bl S_2$ to complete the proof.
\end{proof}
Since $\mathbb P_{\bl p}^{ii}K^{(\l)}(\cdot, w)$ is the reproducing kernel for
$\mathbb P_{\bl p}^{ii}\big (\mb A^{(\l)}(\mb D^n) \big ),$
it can vanish only on a set $X\subseteq \mb D^n$ such that the real dimension of $X$ 
is at most $2n-2$. Also,
\bea \label{pii}
M_{s_i}^{(\bl p ,i)} \mathbb P_{\bl p}^{ii}K^{(\l)}(\cdot, w) = \overline{s_i(w)} \mathbb P_{\bl p}^{ii}K^{(\l)}(\cdot, w),\eea
and therefore,
 $\bl s(\mathbb D^n \setminus X) \subseteq \sigma(M_{s_1}^{(\bl p, i)}, \ldots , M_{s_n}^{(\bl p, i)}).$  Following the usual convention, set $\Gamma_n = {\rm clos}({\mathbb G}_n)$ and note that $\Gamma_n = \bl s\big({\rm clos}{(\mb D^n})\big).$
\begin{thm}
The Taylor  joint spectrum of the $n$-tuple $(M_{s_1}^{(\bl p, i)}, \ldots , M_{s_n}^{(\bl p, i)})$ is  $\Gamma_n.$
\end{thm}
\begin{proof}
From Proposition \ref{Joe}, it follows that $\sigma(M_{s_1}^{(\bl p, i)}, \ldots , M_{s_n}^{(\bl p, i)} ) \subseteq  \sigma (M_{s_1}, \ldots , M_{s_n}).$
The Taylor functional calculus shows that  $\sigma (M_{s_1}, \ldots , M_{s_n})= \bl s(\sigma(M_1, \ldots , M_n) = \Gamma_n.$  Thus we have
$$\mathbb G_n \setminus \bl s(X) \subseteq \bl s(\mathbb D^n \setminus X) \subseteq \sigma(M_{s_1}^{(\bl p, i)}, \ldots , M_{s_n}^{(\bl p, i)} ) \subseteq \Gamma_n.$$
Since ${\rm clos}\big({\mathbb G_n\setminus \bl s(X)}\big)=\Gamma_n$ and  the spectrum is compact, the proof is complete.
\end{proof}
The computation of the Taylor joint spectrum has some immediate applications. Commuting $n$ -tuples of joint weighted shifts are discussed in \cite{JL}. They have shown (see \cite[Corollary 3]{JL}), among other things, that the spectrum of a joint weighted shift must be Reinhardt (invariant under the action of the torus group). It is easy to see  that $\Gamma_n$ is not Reinhardt.  Indeed $(1,\tfrac{1}{2}, \dots , 0)$ is in $\Gamma_n$ while $(1, -\tfrac{1}{2}, 0, \ldots , 0)$ is not in $\Gamma_n.$
This follows from the observation that $(\mu_1, \ldots , \mu_k, 0, \ldots , 0)$ is in $\Gamma_n$ if and only if $(\mu_1, \ldots , \mu_k)$  is in $\Gamma_k.$ 
therefore we have proved the following Corollary.
\begin{cor}
The $n$- tuple  $\big(M_{s_1}^{(\bl p, i)}, \ldots , M_{s_n}^{(\bl p, i)}\big)$ is not unitarily equivalent to any joint weighted shift.
\end{cor}
Let $X\subseteq \mathbb C^n$ be a polynomially convex set. A commuting $n$ - tuple $\bl T$  of operators is said to admit $X$ as a spectral set if $\|p(\bl T)\| \leq \|p\|_{\infty, X}:=\sup\{|p(\bl z)| : \bl z \in X\}.$  In the particular case of $X=\Gamma_n,$ such a   commuting $n$-tuple $\bl T$ is said to be a $\Gamma_n$-contraction. Since the restriction of a $\Gamma_n$-contraction to a reducing subspace is again a $\Gamma_n$-contraction, the proof of the following theorem is evident  from \cite[Proposition 2.13 and Corollary 3.11]{BS}.

\begin{thm}\label{thmB}
The commuting $n$-tuple $ (M_{s_1},\ldots,M_{s_n})$  acting on the Hilbert space $\mb P_{\bl p}^{ii}\big(\mb A^{(\l)}(\mb D^n)\big)$ is a $\Gamma_n$-contraction for every partition  $\bl p$ of $n,$ $1\leq i \leq \chi_{\bl p}(1)$ and all $\l\geq 1.$
\end{thm}
\begin{rem} It is observed in \cite[p.  47]{AY}  that the Taylor joint spectrum of a $\Gamma_2$-contraction is a subset of $\Gamma_2.$ This is easily seen to be true of a $\Gamma_n$-contraction using polynomial convexity of $\Gamma_n.$ Hence the Taylor joint spectrum of the commuting $n$-tuple  $\big(M_{s_1}^{(\bl p, i)}, \ldots , M_{s_n}^{(\bl p, i)}\big)$
is contained in $\Gamma_n.$  Here we emphasize that the  $n$-tuple $\big(M_{s_1}^{(\bl p, i)}, \ldots , M_{s_n}^{(\bl p, i)}\big)$ is not only a $\Gamma_n$-contraction but admits its spectrum $\Gamma_n$ as a spectral set.
\end{rem}
\section{inequivalence}
Having obtained the decomposition \eqref{finedecomp} and having shown
that each $\mb P_{\bl p}^{ii}\big(\mb A^{(\l)}(\mb D^n)\big)$ is a reducing sub-module (Corollary \ref{fd}) over the ring of symmetric polynomials
$\mb C[\bl z]^{\mathfrak S_n}$ of the Hilbert module $\mb A^{(\l)}(\mb D^n)$, it is natural to ask whether these sub-modules are inequivalent for distinct pairs ($\bl p,$ $i$) of a partition $\bl p$ of $n$ and $i,$ $1\leq i \leq \chi_{\bl p}(1)$.  The following theorem provides a partial answer.
\begin{thm}\label{m1}
If $\bl p$ and $\bl q$ are two partitions of $n$ such that $\chi_{\bl p}(1)\not = \chi_{\bl q}(1),$ then
\begin{enumerate}
\item[(a)] the sub-modules  $\mb P_{\bl p}^{ii}\big(\mb A^{(\l)}(\mb D^n)\big)$ and $\mb P_{\bl q}^{jj}\big(\mb A^{(\l)}(\mb D^n)\big)$ are not equivalent for any $i,j,$ $1\leq i \leq \chi_{\bl p}(1)$ and $1\leq j \leq \chi_{\bl q}(1).$
\item[(b)] the sub-modules  $\mb P_{\bl p}\big(\mb A^{(\l)}(\mb D^n)\big)$ and $\mb P_{\bl q}\big(\mb A^{(\l)}(\mb D^n)\big)$ are not equivalent.
\end{enumerate}
\end{thm}
\begin{proof}
From Corollary \ref{fd}, it follows that
\begin{equation*} \bigcap_{k=1}^n\ker\big(M^{(\bl p)}_{s_k}-s_k(\bl w)\big)^*=\bigoplus_{i=1}^{\chi_{\bl p}(1)}\bigcap_{k=1}^n\ker\big(M^{(\bl p, i)}_{s_k}-s_k(\bl w) \big)^*.
\end{equation*}
Arguments similar to the ones given in the proof of Lemma \ref{projPii} applied to the sub-modules $\mb P_{\bl p}^{ii}\big(\mb A^{(\l)}(\mb D^n)\big)$ shows that $\cap_{k=1}^n\ker\big(M^{(\bl p, i)}_{s_k}-s_k(\bl w) \big)^*$ are isomorphic for all $i,$
$1\leq i \leq \chi_{\bl p}(1).$ Thus $\dim\cap_{k=1}^n\ker\big(M^{(\bl p, i)}_{s_k}-s_k(\bl w) \big)^* =\chi_{\bl p}(1)$ for all $i.$ From the proof of Theorem \ref{free}, it follows that each of the sub-modules $\mb P_{\bl p}^{ii}\big(\mb A^{(\l)}(\mb D^n)\big)$ is locally free of rank $\chi_{\bl p}(1)$  on $\mathbb G_n\setminus \bl s(\mathcal Z).$ The rank being an invariant for locally free Hilbert modules, the proof of (a) is complete.  The proof of  (b) follows
from Corollary \ref{cor2.13}.\end{proof}
The theorem leaves open the question of equivalence when $\chi_{\bl p}(1) = \chi_{\bl q}(1).$
%
%
%
While we are not able to settle this question in its entirety, we answer it in the important special case of $\chi_{\bl p}(1) = 1 = \chi_{\bl q}(1),$ or equivalently, $\bl p = (n)$ and $\bl q = (1, \ldots , 1)$ since one dimensional representations of $\mathfrak S_n$ are the trivial and the sign representation.

We begin by setting up some notation which will be useful in the discussion to follow.
The length $\ell(\bl p)$ of a partition $\bl p$ of $n$ is the number of positive summands of $\bl p.$
For a positive integer $n,$ we define the following two subsets of $\mb Z_+^n:=\{(m_1, \ldots , m_n) \in \mathbb Z^n: m_1, \ldots , m_n \geq 0\}$:
\Bea
[n]=\{\bl m\in \mb Z_+^n:m_i\geq m_j \mbox{~for~} i<j\} \mbox{~and~}[\![n]\!]=\{\bl m\in \mb Z_+^n:m_i> m_j \mbox{~for~} i<j\}.
\Eea
 If  $\bl p\in [\![n]\!],$  then we can write $\bl p=\bl m+\bl\delta,$ where $\bl m\in[n]$ and $\bl\delta=(n-1,n-2,\ldots, 1,0).$ So,
\Bea
[\![n]\!]=\{\bl m+\bl\delta: \bl m\in[n]\}.
\Eea

Recall from equation \eqref{od} that for a partition $\bl p$ of $n,$ the linear map $\mb P_{\bl p}:\mb A^{(\l)}(\mb D^n)\to \mb A^{(\l)}(\mb D^n)$ by
\bea\label{proj}
\mb P_{\bl p}f=\frac{\chi_{\bl p}(1)}{n!}\sum_{\sigma\in\mathfrak S_n}\ov{\chi_{\bl p}(\sigma)}f\circ\sigma^{-1},
\eea
where $\chi_{\bl p}$ is the character of the representation corresponding to the partition $\bl p$ of $n.$
Choosing the partition $\bl p$ of $n$ to be $(n):=(n,0,\ldots,0)$ in Equation \eqref{proj}, it is easy to see that
\Bea
\mb P_{(n)}\big (\mb A^{(\l)}(\mb D^n) \big )=\{f\in \mb A^{(\l)}(\mb D^n):f\circ\sigma^{-1}=f \mbox{~for~} \sigma\in\mathfrak S_n\},
\Eea
that is, $\mb P_{(n)}\big (\mb A^{(\l)}(\mb D^n)\big )$ consists of symmetric functions in $\mb A^{(\l)}(\mb D^n).$  Thus $\mb A^{(\l)}_{\rm sym}(\mb D^n)=\mb P_{(n)}\big (\mb A^{(\l)}(\mb D^n)\big ).$ In view of \cite[Equation (3.1)]{BS},  the following proposition is a particular case of \cite[Proposition 3.6]{BS} for $\bl p=(n).$
\begin{prop}\label{sker}
 The reproducing kernel $K^{(\l)}_{\rm  sym}$ of $\mb A^{(\l)}_{\rm sym}(\mb D^n)$ is given explicitly by
the formula:
\Bea
K^{(\l)}_{\rm sym}(\bl z, \bl w)= \frac{1}{n!}\mr{per} \Big (\big (\!\!\big ((1-z_j\bar w_k)^{-\l} \big )\!\!\big )_{j,k=1}^n\Big
),\,\, \bl z,\bl w\in\mb D^n,
\Eea
where $\mbox{\rm per} \Big (\big (\!\! \big ( a_{ij} \big )\! \!\big )_{i,j=1}^n \Big ) = \sum_{\sigma \in \mathfrak S_n} \prod_{k=1}^n a_{k \sigma(k)}.$
\end{prop}
The Hilbert space $\mb A^{(\l)}_{\rm sym}(\mb D^n)$ can be thought of as a space of functions defined on the symmetrized polydisc $\mb{G}_n$ as follows. Recall that $\bl s$ is the symmetrization map and  note that
\Bea
\mb A^{(\l)}_{\rm sym}(\mb D^n)=\{f\in \mb A^{(\l)}(\mb D^n): f=g\circ \bl s \mbox{  ~for some~} g: \mb{G}_n \lo \mb{C}\mbox{  ~ holomorphic~} \}.
\Eea
Let
 \Bea
\mathcal{H}^{(\l)}(\mb{G}_n) := \{g: \mb{G}_n \lo \mb{C} \mbox{  ~is holomorphic~}: g\circ \bl s \in \mb A^{(\l)}(\mb D^n)  \}.
\Eea
The inner product on $\mathcal{H}^{(\l)}(\mb{G}_n)$ is given by
$\inner{f_1}{f_2}_{\mathcal{H}^{(\l)}(\mb{G}_n)}$ := $\inner{f_1 \circ s}{f_2\circ s}_{\mb A^{(\l)}(\mb D^n)}.$ Now, the following corollary is immediate from Proposition \ref{sker}.

\begin{cor}\label{sg}
 The reproducing kernel $K^{(\l)}_{\mb G_n}$ of $\m H^{(\l)}(\mb G_n)$ is given explicitly by
the formula:
$$K^{(\l)}_{\mb G_n}\big(\bl s(\bl z), \bl s(\bl w)\big)= \frac{1}{n!}\mr{per} \Big (\big (\!\!\big ((1-z_j\bar w_k)^{-\l} \big )\!\!\big )_{j,k=1}^n\Big
),\,\, \bl z,\bl w\in\mb D^n.$$
\end{cor}
Choosing the partition $\bl p$ of $n$ to be $(1^n):=(1,\ldots,1)\in [n],$ we see that
\Bea
\mb P_{(1^n)}\big (\mb A^{(\mu)}(\mb D^n)\big )=\{f\in \mb A^{(\mu)}(\mb D^n):f\circ\sigma^{-1}={\rm sgn} (\sigma)f \mbox{~for~} \sigma\in\mathfrak S_n\}.
\Eea
Since $\mb P_{(1^n)}\big (\mb A^{(\mu)}(\mb D^n)\big )$ consists  of anti-symmetric functions, therefore $\mb A^{(\mu)}_{\rm anti}(\mb D^n) =\mb P_{(1^n)}\big (\mb A^{(\mu)}(\mb D^n)\big ).$ Appealing to  \cite[Proposition 3.8]{BS} for $\bl p=(n)$ and $\bl p=(1^n),$ we have a particular case of \cite[Proposition 3.8]{BS}, which we record below for future reference.
\begin{lem}
The Hilbert spaces  $\mb A^{(\l)}_{\rm sym}(\mb D^n)$ and  $\mb A^{(\mu)}_{\rm anti}(\mb D^n)$ are Hilbert modules over $\mb C[\bl z]^{\mathfrak S_n},$ under its natural action for $\l,\mu>0$ and $n \geq 2.$
\end{lem}
The theorem below provides an affirmative answer to the question we raised in the beginning of this section.
\begin{thm}\label{m}
The Hilbert modules $\mb A^{(\l)}_{\rm sym}(\mb D^n)$ and $\mb A^{(\l)}_{\rm anti}(\mb D^n)$ over $\mb C[\bl z]^{\mathfrak S_n}$ are not  equivalent for any $\l > 0$ and $n \geq 2.$
\end{thm}

We recall that $\mb C[\bl z]^{\mathfrak S_n}=\mb C[s_1,\ldots, s_n].$ In view of this fact $\m H^{(\l)}(\mb G_n)$ is a Hilbert module over $\mb C[\bl z]^{\mathfrak S_n},$ under the natural action of $\mb C[\bl z]^{\mathfrak S_n}.$  Consider the map from $\mathcal{H}^{(\l)}(\mb{G}_n)$ to $\mb A^{(\l)}_{\rm sym}(\mb D^n)$  defined by $f \mapsto f \circ \bl s$ and note that it is a unitary map  which intertwines the $n$-tuple  $(M_{s_1},M_{s_2},\ldots, M_{s_n})$ of multiplication operators by the coordinate functions $s_1,\ldots, s_n$ and the tuple $(M_{s_1(\bl z)},M_{s_2(\bl z)},\ldots, M_{s_n(\bl z)}),$  where $s_i(\bl z)$ is the $i$-th elementary symmetric function in $z_1,\ldots, z_n$ for $i=1,\ldots,n.$  Therefore, there is a unitary module map  between the Hilbert modules $\m H^{(\l)}(\mb G_n)$ and $\mb A^{(\l)}_{\rm sym}(\mb D^n)$ over $\mb C[\bl z]^{\mathfrak S_n}.$ We record this observation in the form of a lemma.
\begin{lem}\label{sue}
 For $\l>0,$ the Hilbert modules $\mathcal{H}^{(\l)}(\mb{G}_n)$ and $\mb A^{(\l)}_{\rm sym}(\mb D^n)$
are  equivalent as modules over $\mb C[\bl z]^{\mathfrak S_n}.$
\end{lem}

Now we describe the weighted Bergman space on the symmetrized polydisc  $\mb G_n$  as a module over $\mb C[\bl z]^{\mathfrak S_n}.$
For $\mu>1,$ let $dV^{(\mu)}$ be the probability measure $\big(\frac{\mu-1}{\pi}\big)^n\Big(\prod_{i=1}^n(1-r_i^2)^{\mu-2}r_idr_id\theta_i\Big)$ on the polydisc $\mb D^n.$ Let $dV^{(\mu)}_{\bl s}$ be the measure on the symmetrized polydisc $\mb G_n$  obtained by the change of variables formula  \cite[p. 106]{B}:
\bea\label{cov}
\int_{\mb G_n}f dV^{(\mu)}_{\bl s}=\frac{1}{n!}\int_{\mb D^n}(f\circ\bl s)\vert J_{\bl s}\vert^2dV^{(\mu)}, \,\, \mu>1,
\eea
where $ J_{\bl s}(\bl z)=\Delta (\bl z)$ is the complex jacobian of the symmetrization map $\bl s.$
The weighted Bergman space $ \mb{A}^{(\mu)} (\mb G_n),\, \mu>1,$ on the  symmetrized polydisc $\mb G_n$ is the subspace of $L^2(\mb G_n, dV^{(\mu)}_{\bl s})$ consisting of holomorphic functions. For $\mu>1,$ consider the map  $\Gamma: \mb{A}^{(\mu)} (\mb G_n)\to \mb{A}^{(\mu)} (\mb{D}^n)$ defined by
\bea\label{iso}
\Gamma f=\frac{1}{ \sqrt{n!}} J_{\bl s}(f\circ \bl s), \,\, f\in \mb{A}^{(\mu)} (\mb G_n) .
\eea
It follows from Equation \eqref{cov} that $\Gamma$ is an isometry onto $\mb{A}^{(\mu)}_{\rm anti} (\mb{D}^n)$ \cite[p. 2363]{MSZ}.
One can easily check that $\Vert \bl z^{\bl m}\Vert^2_{ \mb{A}^{(\mu)} (\mb{D}^n)} =\big \Vert  z_1^{ m_1}\ldots z_n^{m_n}\big \Vert^2_{ \mb{A}^{(\mu)} (\mb{D}^n)}=\frac{m_1!\ldots m_n!}{(\mu)_{m_1}\ldots (\mu)_{m_n}}.$ For a partition $\bl m=(m_1,\ldots,m_n)\in [\![n]\!],$ put $a_{\bl m}(\bl z)=a_{\bl p+\bl\delta}(\bl z)=\det\Big((\!(z_i^{m_j})\!)_{i,j=1}^n\Big),$ where $ \bl p\in[n]$  and $\bl m=\bl p+\bl\delta.$ The norm of $a_{\bl m}$ in $ \mb{A}^{(\mu)} (\mb{D}^n)$ is easily calculated using orthogonality of distinct monomials in $ \mb{A}^{(\mu)} (\mb{D}^n):$
\Bea
\Vert a_{\bl m}\Vert^2_{ \mb{A}^{(\mu)} (\mb{D}^n)}=\Big \Vert\sum_{\sigma\in\mathfrak S_n}{\rm sgn}(\sigma)\prod_{k=1}^nz_k^{m_{\sigma(k)}}\Big  \Vert^2_{ \mb{A}^{(\mu)} (\mb{D}^n)}
=\sum_{\sigma\in\mathfrak S_n}\big \Vert \prod_{k=1}^nz_k^{m_{\sigma(k)}}\big \Vert_{\mb{A}^{(\mu)} (\mb{D}^n)}^2=\frac{n!\bl m!}{(\mu)_{\bl m}},
\Eea
where $\bl m!=\prod_{j=1}^nm_j!$ and $(\mu)_{\bl m}=\prod_{j=1}^n(\mu)_{m_j}.$ Here $(\mu)_{m_j}$ is the Pochhammer symbol $(\mu)_{m_j}=\mu(\mu+1)\ldots (\mu+m_j-1).$ Putting $c_{\bl m}=\sqrt{\frac{(\mu)_{\bl m}}{n!\bl m!}},$ it follows from \cite[p. 2364]{MSZ} that
\Bea
\{e_{\bl m}=c_{\bl m}a_{\bl m}:\bl m\in [\![n]\!]\}
\Eea
is an orthonormal basis of $\mb A^{(\mu)}_{\rm anti}(\mb D^n).$

 The determinant function $a_{\bl p+\bl\delta}$ is a polynomial and is  divisible by each of  the differences  $z_i-z_j, 1\leq i<j\leq n$ and hence by the product
\Bea
\prod_{ 1\leq i<j\leq n}^n(z_i-z_j)=\det\Big((\!(z_i^{n-j})\!)_{i,j=1}^n\Big)=a_{\bl\delta}(\bl z)=\Delta(\bl z).
\Eea
  For $\bl p\in[n],$ the quotient $ S_{\bl p}:=\frac{a_{\bl p+\bl\delta}}{a_{\bl\delta}},$ is therefore well-defined and is called the Schur polynomial \cite[p. 454]{FH}. For $ \bl p\in[n]$  and $\bl m=\bl p+\bl\delta,$ recall that $c_{\bl m}=c_{\bl p+\bl\delta}=\sqrt{\frac{(\mu)_{\bl p+\bl\delta}}{n!(\bl p+\bl\delta)!}},$ now it follows from  Equation \eqref{iso} that
\Bea
\Gamma \Bigg(\sqrt{\frac{(\mu)_{\bl p+\bl\delta}}{(\bl p+\bl\delta)!}}S_{\bl p}\Bigg)=\Gamma\Big(\sqrt{n!} c_{\bl p+\bl\delta}S_{\bl p}\Big)=c_{\bl p+\bl\delta}a_{\bl p+\bl\delta}=c_{\bl m}a_{\bl m},\,\, \bl m\in [\![n]\!].
\Eea
 Since the map $\Gamma:\mb{A}^{(\mu)} (\mb G_n)\to \mb{A}^{(\mu)}_{\rm anti} (\mb{D}^n)$ defined by Equation \eqref{iso}  is a unitary  \cite[p. 2363]{MSZ}, the set
\Bea
\{\g_{\bl p}S_{\bl p}:\bl p\in[n]\}, \mbox{~where~} \g_{\bl p}=\sqrt{\frac{(\mu)_{\bl p+\bl\delta}}{(\bl p+\bl\delta)!}},
\Eea
is an orthonormal basis for $\mb A^{(\mu)}(\mb G_n).$ Hence we have the following proposition,

\begin{prop}\label{srk}
 The reproducing kernel $\mathbf{B}_{\mb{G}_n}^{(\mu)}$ for $\mb{A}^{(\mu)}(\mb{G}_n)$ is given by
\bea\label{rk}
\mathbf{B}_{\mb{G}_n}^{(\mu)}\big(\bl s(\bl z),\bl s(\bl w)\big)
=\sum_{\bl p \in [n] } \g_{\bl p}^2S_{\bl p}(\bl z)\overline{S_{\bl p}(\bl w)}, \,\, \bl z,\bl w\in\mb D^n, \mu>1.
\eea

\end{prop}

From \cite[p. 2363]{MSZ}, it follows that $ \mb{A}^{(\mu)}_{\rm anti} (\mb{D}^n)$ and the weighted Bergman module $\mb{A}^{(\mu)}(\mb{G}_n)$ are unitarily  equivalent as modules over $\mb C[\bl z]^{\mathfrak S_n}$ for $\mu>1.$ The limiting case $\mu=1,$  is discussed in \cite[p. 2367]{MSZ}. It is not difficult to show that the function $\mathbf B^{(\mu)}_{\mb G_n}:\mb G_n\times\mb G_n\to \mb C,$ defined by the Equation \eqref{rk}, is positive definite for $\mu>0.$ For  $0<\mu<1,$ let $\mb{A}^{(\mu)} (\mb G_n)$ be the Hilbert space of holomorphic functions having $\mathbf B^{(\mu)}_{\mb G_n}$ as its reproducing kernel. If we assume that the set $\{S_{\bl p}\}_{\bl p\in [n]}$ is orthogonal in $\mb{A}^{(\mu)} (\mb G_n)$ and $\Vert S_{\bl p}\Vert^2=\frac{(\bl p+\bl\delta)!}{(\mu)_{\bl p+\bl\delta}},$ then it is easy to verify that the  injective linear map $\Gamma: \mb{A}^{(\mu)} (\mb G_n)\to \mb{A}^{(\mu)} (\mb{D}^n)$ defined in Equation \eqref{iso} is an isometry. By similar arguments as in the case $\mu > 1,$ we reach the desired conclusion for $0<\mu<1$ as well. This observation is recorded in the following Lemma.

\begin{lem}\label{aue}
For $\mu>0,$ the Hilbert modules $ \mb{A}^{(\mu)} (\mb{G}_n)$ and $\mb{A}^{(\mu)}_{\rm anti} (\mb{D}^n)$ are  equivalent, as modules over $\mb C[\bl z]^{\mathfrak S_n}.$
\end{lem}

In view of Lemma \ref{sue} and Lemma \ref{aue}, proving Theorem \ref{m} boils down to  proving the following theorem.

\begin{thm}\label{Gn}
The Hilbert modules  $\mb A^{(\l)}(\mb G_n)$  and $\m H^{(\l)}(\mb G_n)$  over $\mb C[\bl z]^{\mathfrak S_n}$ are not equivalent for any $\l > 0$ and $n \geq 2.$
\end{thm}
To prove this theorem, we recall the notion of a normalized kernel from \cite{CS}. Let $\Omega\subseteq \mb C^n$ be domain.  A kernel function $K:\Omega\times\Omega\to\mb C$ is said to be normalized at $w_0\in\Omega$ if $K(z, w_0)=1$ for $z\in\Omega_0,$ where $\Omega_0\subseteq\Omega,$ is a neighborhood of $w_0.$
We note that $ S_{\bl p}$ is a homogeneous  symmetric polynomial of degree $\vert \bl p\vert:=\sum_{i=1}^np_i ,$ so $S_{\bl 0}\equiv 1$ and $S_{\bl p}(\bl 0)=0$ for $\bl p\neq \bl 0,$ where $\bl 0\in [n]$ with all components equal to $0.$  From Equation \eqref{rk} and the discussion following Proposition \ref{srk}, we see that $\mathbf B^{(\mu)}_{\mb G_n}\big(\bl s(\bl z),\bl 0\big)=\g_{\bl 0}^2=\frac{(\mu)_{\bl\delta}}{\bl\delta!}$ for $\bl z\in\mb D^n$ and $\mu>0.$ We record the following obvious corollary of Proposition \ref{srk} for future reference.
\begin{cor}\label{nork}
The normalized reproducing kernel $\wi{\mathbf B}^{(\mu)}_{\mb G_n}$ for $\mb{A}^{(\mu)}(\mb{G}_n)$ is given by
\bea\label{nrk}
\wi{\mathbf B}_{\mb{G}_n}^{(\mu)}\big(\bl s(\bl z),\bl s(\bl w)\big)
=\frac{{\bl\delta}!}{(\mu)_{\bl\delta}}\sum_{\bl p \in [n] } \g_{\bl p}^2S_{\bl p}(\bl z)\overline{S_{\bl p}(\bl w)}, \,\, \bl z,\bl w\in\mb D^n, \mu>0.
\eea
\end{cor}
It is of independent interest to express the reproducing kernel $\mathbf B_{\mb{G}_n}^{(\mu)}$ in terms of coordinates of $\mb G_n,$ that is, in terms of elementary symmetric polynomials. In order to do that, we  need to introduce some terminologies. To a partition $\bl p=(p_1,\ldots, p_n)\in [n]$ is associated a {\it Young diagram} \cite[Section 4.1]{FH} with $p_i$ boxes in the $i$-th row, the rows of boxes lined up on the left. The {\it conjugate partition}  $\bl p^\i=(p_1^\i,\ldots,p_r^\i)$ to the partition $\bl p$ is defined by interchanging rows and columns in the Young diagram, that is, reflecting the diagram in the $45^\circ$  line. For example, the conjugate partition to the partition $(3,3,2, 1,1)$ is $(5,3,2).$ For the conjugate partition $\bl p^\i=(p_1^\i,\ldots,p_r^\i)$ to $\bl p,$ let us require that $p_r^\i>0$ and call $r$ the length of $\bl p^\i .$ Let us agree to call $s_k$  the $k$-th elementary symmetric polynomial in $n$ variables for $k=0,1,\ldots,$ with the convention that $s_k\equiv 0$ if $k>n.$  We are now ready to state the second of {\it Giambelli's formulas} expressing the Schur polynomials as functions of  elementary symmetric polynomials. Here is Giambelli's second formula  \cite[p. 455]{FH}:
\bea\label{G}
S_{\bl p}=\det\Big((\!(s_{p_i^\i+j-i})\!)_{i,j=1}^r\Big),\,\, \bl p\in [n],
\eea
where $\bl p^\i=(p_1^\i,\ldots,p_r^\i)$ is the conjugate partition to $\bl p.$

Combining Corollary \ref{nork} with the Equation \eqref{G}, we obtain the following theorem.
\begin{thm}\label{BergKerGn}
The normalized reproducing kernel $\wi{\mathbf B}^{(\mu)}_{\mb G_n}$ for $\mb{A}^{(\mu)}(\mb{G}_n)$ is given by
\Bea
\wi{\mathbf B}_{\mb{G}_n}^{(\mu)}(\bl s,\bl t)
=\frac{{\bl\delta}!}{(\mu)_{\bl\delta}}\sum_{\bl p \in [n] } \g_{\bl p}^2\det\Big((\!(s_{p_i^\i+j-i})\!)_{i,j=1}^r\Big)\overline{\det\Big((\!(t_{p_i^\i+j-i})\!)_{i,j=1}^r\Big)},
\Eea
for $\bl s=(s_1,\ldots, s_n), \bl t=(t_1,\ldots, t_n)\in\mb G_n, \mu>0$ and  $\bl p^\i=(p_1^\i,\ldots,p_r^\i)$ is the conjugate partition to $\bl p\in [n].$

\end{thm}

\begin{lem}\label{coe1}
Let $\wi{\mathbf B}^{(\mu)}_{\mb G_n}$ be the  normalized reproducing kernel for $\mb{A}^{(\mu)}(\mb{G}_n).$ Then
\begin{enumerate}
\item[(i)]  the coefficient of  $s_1(\bl z)\overline{s_1(\bl w)}$  in  $\wi{\mathbf B}_{\mb{G}_n}^{(\mu)}\big(\bl s(\bl z ),\bl s(\bl w)\big)$ is $ \frac{\mu+n-1}{n},$
\item[(ii)]  the coefficient of  $s_1(\bl z)^2\overline{s_1(\bl w)^2}$  in  $\wi{\mathbf B}_{\mb{G}_n}^{(\mu)}\big(\bl s(\bl z ),\bl s(\bl w)\big)$ is $\frac{(\mu + n-1)(\mu + n)}{n(n+1)}.$
\end{enumerate}
\end{lem}
\begin{proof}
Since the Schur polynomial  $S_{\bl p}$ is a homogeneous  symmetric polynomial of degree $\vert \bl p\vert:=\sum_{i=1}^np_i ,$  therefore, it is a polynomial in the elementary symmetric polynomials $s_i(\bl z)$ for $i=1,\ldots, n.$  For a fixed $k, q \in \mb{Z}_{+} ,$ the term $s_k(\bl z)^q\overline{s_k(\bl w)^q}$ in $\wi{\mathbf B}_{\mb{G}_n}^{(\mu)}\big(\bl s(\bl z),\bl s(\bl w)\big)$ comes only from the terms which involves $S_{\bl p}(\bl z)\overline{S_{\bl p}(\bl w)}$ in the series for $\wi{\mathbf B}_{\mb{G}_n}^{(\mu)}\big(\bl s(\bl z),\bl s(\bl w)\big)$ in Equation \eqref{nrk}, where   $\bl p=(p_1,\ldots,p_n) \in [n] $ such that  $\sum_{i=1}^{n} p_i =kq.$

To get the coefficient of $s_1(\bl z)\overline{s_1(\bl w)}$ in $\wi{\mathbf B}_{\mb{G}_n}^{(\mu)}\big(\bl s(\bl z),\bl s(\bl w)\big),$  take $\bl p=(1,0,...,0).$ From Equation \eqref{nrk}, the coefficient of  $s_1(\bl z)\overline{s_1(\bl w)}$  in  $\wi{\mathbf B}_{\mb{G}_n}^{(\mu)}\big(\bl s(\bl z ),\bl s(\bl w)\big)$ is
\Bea\label{one2}
\frac{{\bl\delta}!}{(\mu)_{\bl\delta}}\g_{\bl p}^2 =\frac{{\bl\delta}!}{(\mu)_{\bl\delta}}\cdot\frac{(\mu)_{\bl p+\bl\delta}}{(\bl p+\bl\delta)!}= \frac{\mu+n-1}{n},
\Eea
where $\bl p=(1,0,\ldots, 0).$ This proves (i).

Similarly, to obtain the coefficient of $s_1(\bl z)^2\overline{s_1(\bl w)^2} $ in  $\wi{\mathbf B}_{\mathbf{G}_n}^{(\mu)}\big(\bl s(\bl z),\bl s(\bl w)\big),$ we need to consider terms corresponding to $\bl p=(2,0,\ldots,0)$ and $\bl p=(1,1,0,\ldots,0).$  From the Giambelli's formula \eqref{G}, we get
$S_{(2,0,...,0,0)}(\bl z)= ({s_1}^2 - s_2)(\bl z) $ and $S_{(1,1,...,0,0)}(\bl z)=  s_2(\bl z).$

Since $s_1^2$ appears  only in $S_{(2,0,\ldots,0)},$ from Equation \eqref{nrk},  it follows that the coefficient of  $s_1(\bl z)^2\overline{s_1(\bl w)^2}$  in  $\wi{\mathbf B}_{\mb{G}_n}^{(\mu)}\big(\bl s(\bl z ),\bl s(\bl w)\big)$ is
 \Bea\label{onesq2}
\frac{{\bl\delta}!}{(\mu)_{\bl\delta}}\g_{\bl p}^2 =\frac{{\bl\delta}!}{(\mu)_{\bl\delta}}\cdot\frac{(\mu)_{\bl p+\bl\delta}}{(\bl p+\bl\delta)!}= \frac{(\mu + n-1)(\mu + n)}{n(n+1)},
\Eea
where $\bl p=(2,0,\ldots,0).$ This proves (ii).
\end{proof}

Consider the restriction of the action of $\mathfrak{S}_n$ to $\mb{Z}_{+}^{n}.$
Let $\mathfrak S_n\bl m$ denote the orbit of $\bl m\in\mb Z_+^n.$
If  $\bl m\in [n]$ has $k(\leq n)$ distinct components, that is, there are $k$ distinct  non-negative integers $m_1>\ldots> m_k$ such that
\Bea
\bl m=(m_1,\ldots,m_1, m_2,\ldots,m_2,\ldots,m_k,\ldots,m_k),
\Eea
where each $m_i$ is repeated $\a_i$ times, for $i=1,\ldots, k,$ then $\bl \a=(\a_1,\ldots, \a_k)$ is said to be the  {\it multiplicity} of $\bl m\in [n].$ For any $\bl m\in\mb Z_+^n$ the components of $\bl m$ can be arranged in the decreasing order to obtain, say, $\widetilde{\bl m}\in [n].$ We say that $\bl m\in\mb Z_+^n$ is of {\it multiplicity} $\bl \a=(\a_1,\ldots, \a_k)$ if $\widetilde{\bl m}$  has multiplicity $\bl \a$. In particular, the elements of $[\![n]\!]$ are of multiplicity $(1^n),$ that is, $1$ occurs $n$-times.

We recall that the number of distinct $n$-letter words with $k$ distinct letters is  $\frac{n!}{\bl\a!}=\frac{n!}{\a_1!\ldots \a_k!},$
where the $k$ distinct letters $a_1,\ldots, a_k$ are repeated $\a_1,\ldots, \a_k$ times, respectively $(\a_1+\ldots+\a_k=n).$ In other words, for a fixed $\bl m\in\mb Z_+^n$, we have $\vert \mathfrak  S_n\bl m\vert=\frac{n!}{\bl\a!},$ where $\vert X\vert$ denotes the cardinality of a set $X.$ Let   $\mb Z_+^n/\mathfrak S_n$ denote the set of all orbits of $\mb Z_+^n$ under the action of  $\mathfrak S_n.$ We record the following as a lemma for later use.






\begin{lem}\label{rep}
The set $\mb Z_+^n/\mathfrak S_n$  is in one-one correspondence with the set $[n].$
\end{lem}
\begin{proof} First, we prove that each $\mathfrak S_n$ orbit of $\mathbb Z_+^n$ has exactly one $n$-tuple in decreasing order. To see this, observe that each orbit contains an $n$-tuple in decreasing order, and hence enough to prove it is unique. Suppose there are two $n$-tuples in decreasing order, say $\bl m, \bl m^\prime,$ in the same orbit. Since a permutation only changes the position of a component, it follows that all $n$-tuples in an orbit have the same multiplicity. Therefore the multiplicity of $\bl m$ and $\bl m^\prime$ is the same and hence $\bl m =\bl m^\prime$.
Note that each element in $[n]$ is in some orbit and hence the proof is complete.
\end{proof}

Consider the  monomial symmetric polynomials \cite[p. 454]{FH}
 $$\mb{M}_{\bl m}(\bl z) = \sum_{\bl\beta}\bl z^{\bl\beta},$$
 where the sum is over all distinct permutations $\bl \beta=(\beta_1,\beta_2,...,\beta_n)$ of $\bl m\in[n]$ and $\bl z^{\bl\beta} = z_1^{\beta_1} z_2^{\beta_2}...z_n^{\beta_n}.$ This definition of ${\mb M}_{\bl m}$ makes sense for $\bl m\in\mb Z_+^n$ as well and we use  it in the sequel.   Observe that $ \mathfrak  S_n{\bl m}$ is the set of all distinct permutations of ${\bl m},$  so,
 \bea\label{mono}
 \hspace{1  cm}{\mb M}_{{\bl m}}(\bl z)=\sum_{\bl\b\in\mathfrak S_n{{\bl m}}}\bl z^{\bl \b} = \mb{M}_{{\bl m}^\i}(\bl z)\mbox{~for~}{\bl m},{\bl m}^{\i} \in \mathfrak S_n{{\bl m}}.
\eea
 The following lemma that gives us an expression for the reproducing kernel ${K}_{\mb{G}_n}^{(\l)}$ for $\mathcal{H}^{(\l)}(\mb{G}_n)$ will play a significant role in the sequel.

 \begin{lem}\label{km}
 The reproducing kernel ${K}_{\mb{G}_n}^{(\l)}$ for $\mathcal{H}^{(\l)}(\mb{G}_n)$ is given by the formula:
\Bea
{K}_{\mb{G}_n}^{(\l)}\big(\bl s(\bl z),\bl s(\bl w)\big) = \frac{1}{n!}\sum_{\bl m \in [n]} \frac{{\bl\alpha!(\l)_{\bl m}}}{\bl m!} \mb{M}_{\bl m}(\bl z) \overline{\mb{M}_{\bl m}(\bl w)},\, \bl z,\bl w\in\mb D^n,
\Eea
 where $\bl m$ is of multiplicity $\bl\a$.
 \end{lem}

 \begin{proof}
  If  ${{\bl m}}\in [n]$ is of multiplicity $\bl\a$, then ${\mb M}_{\bl m}(\bl z)$ is the sum of  $\vert \mathfrak S_n\bl m\vert=\frac{n!}{\bl\alpha!}$ distinct monomials. We then observe that
\bea\label{pex}
{\rm per} \Big((\!(z_i^{m_j})\!)_{i,j=1}^{n}\Big) =\d_{\sigma\in\mathfrak S_n}\prod_{i=1}^nz_i^{m_{\sigma(i)}}=\d_{\sigma\in\mathfrak S_n}\prod_{i=1}^nz_i^{m_{\sigma^{-1}(i)}}=\d_{\sigma\in\mathfrak S_n}{\bl z}^{\sigma\cdot\bl m} ,
\eea
 is the sum of $n!$  monomials,  from which exactly $ \vert\mathfrak  S_n{{\bl m}}\vert=\frac{n!}{\bl\a!}$   are distinct (since there can be only $\frac{n!}{\bl\a!}$ distinct permutations of a $\bl m\in [n]$ with multiplicity ${\bl\a}$). So, each distinct term must be repeated $\bl\a!$ times. Thus, from equation \eqref{mono} we conclude that

  \bea\label{mpe}
 {\rm per} \Big((\!(z_i^{m_j})\!)_{i,j=1}^{n}\Big) = \bl\alpha!\mb{M}_{\bl m^\i}(\bl z), \,\, \mbox{~for~any~} \bl m^\i\in \mathfrak S_n{{\bl m}}.
 \eea

Since $\mb Z_+^n$ is the disjoint union of its $\mathfrak S_n$-orbits, from Lemma \ref{rep} we have
\bea\label{union}
\mb Z_+^n={\cup}_{\bl m\in [n]}\mathfrak S_n\bl m.\eea

Therefore, from Corollary \ref{sg}, we have
\Bea
{K}_{\mb{G}_n}^{(\l)}\big(\bl s(\bl z),\bl s(\bl w)\big)
&=&\frac{1}{n!} \mr{per}\Big ( \big (\! \big ((1-z_j\bar w_k)^{-\l} \big )\!\big )_{j,k=1}^n\Big)\\
&=&\frac{1}{n!} \d_{\sigma\in  \mathfrak S_n}\displaystyle\prod_{i=1}^n(1-z_i\bar w_{\sigma(i)})^{-\l}\\
&=&\frac{1}{n!} \d_{\sigma\in  \mathfrak S_n}\d_{\bl m \in \mb Z_+^n}\frac{(\l)_{\bl m}}{\bl m!}\,\displaystyle\prod_{i=1}^n z_i^{m_i}\displaystyle\prod_{i=1}^n\bar w_{\sigma(i)}^{m_i}\\
&=&\frac{1}{n!} \d_{\bl m \in \mb Z_+^n}\frac{(\l)_{\bl m}}{\bl m!}\,\displaystyle\prod_{i=1}^n z_i^{m_i}\d_{\sigma\in  \mathfrak S_n}\displaystyle\prod_{i=1}^n\bar w_{\sigma(i)}^{m_i}\\
&=&\frac{1}{n!} \sum_{\bl m\in \mb Z_+^n} \frac{(\l)_{\bl m}}{\bl m!} \displaystyle\prod_{i=1}^n z_i^{m_i}{{\rm per}\Big((\!({\bar w}_i^{m_j})\!)_{i,j=1}^{n}}\Big)\\
&=&\frac{1}{n!} \d_{\bl m\in [n]} \d_{\bl m^\i\in \mathfrak S_n{{\bl m}}}\frac{(\l)_{\bl m^\i}}{\bl m^\i!} \displaystyle\prod_{i=1}^n z_i^{m_i^\i}{{\rm per}\Big( (\!({\bar w}_i^{m_j^\i})\!)_{i,j=1}^{n}}\Big) \mbox{~(using~\eqref{union})}\\
&=&\frac{1}{n!} \d_{\bl m\in [n]} \frac{(\l)_{\bl m}}{\bl m!}{{\rm per}\Big((\!({\bar w}_i^{m_j})\!)_{i,j=1}^{n}}\Big)\d_{\bl m^\i\in \mathfrak S_n{{\bl m}}}\displaystyle\prod_{i=1}^n z_i^{m_i^\i} \\
&=&\frac{1}{n!} \d_{\bl m\in [n]} \frac{(\l)_{\bl m}}{\bl m!} \bl\alpha! \overline{\mb{M}_{\bl m}(\bl w)}\d_{\bl m^\i\in \mathfrak S_n{{\bl m}}}\bl z^{\bl m^\i} \mbox{~(using~\eqref{mpe})}\\
&=&\frac{1}{n!}\sum_{\bl m \in [n]} \frac{{\bl\alpha!(\l)_{\bl m}}}{\bl m!} \mb{M}_{\bl m}(\bl z) \overline{\mb{M}_{\bl m}(\bl w)},
\Eea
where the last equality follows from  Equation \eqref{mono}.
\end{proof}

\begin{rem}
One could also write the reproducing kernel in terms of permanent using the equations \eqref{mpe} and \eqref{union} and the equality $\vert \mathfrak  S_n\bl m\vert=\frac{n!}{\bl\a!}$, as follows:
\Bea
{K}_{\mb{G}_n}^{(\l)}\big(\bl s(\bl z),\bl s(\bl w)\big) &=& \frac{1}{n!} \d_{\bl m\in [n]} \d_{\bl m^\i\in \mathfrak S_n{{\bl m}}}
\big (\frac{n!}{\bl\a!}\big )^{-1}\frac{{\bl\alpha!(\l)_{\bl m^\i}}}{\bl m^\i!}\frac{1}{\bl\a!}\mr{per}\Big ((\!( z_i^{m^\i_j} )\!)_{i,j=1}^n \Big)\frac{1}{\bl\a!}\mr{per}\Big((\!( \bar w_i^{m^\i_j} )\!)_{i,j=1}^n\Big)\\
&=&\frac{1}{(n!)^2}\sum_{\bl m\in\mb Z_+^n} \frac{(\l)_{\bl m}}{\bl m!} \mr{per}\Big ((\!( z_i^{m_j} )\!)_{i,j=1}^n \Big)\mr{per}\Big((\!( \bar w_i^{m_j} )\!)_{i,j=1}^n\Big),
\Eea
for $\bl z,\bl w\in\mb D^n.$
\end{rem}

We note that the kernels $B_{\mathbb G_n}^{(\l)}$ and $K_{\mathbb G_n}^{(\l)}$ are defined on all  of $\mathbb G_n.$ Hence the Hilbert modules $\mathbb P_{(n)} \big (\mathbb A^{(\l)}(\mathbb  D^n)\big )$ and $\mathbb P_{(1,\ldots ,1)} \big (\mathbb A^{(\l)}(\mathbb  D^n)\big )$ are locally free on  all of $\mathbb G_n$ strengthening our earlier assertion (Corollary \ref{cor2.13}) that they are locally free only on $\mathbb G_n \setminus \bl s(\mathcal Z).$ Thus we have proved the following Corollary.
\begin{cor}
The Hilbert modules $\mathbb P_{(n)} \big (\mathbb A^{(\l)}(\mathbb  D^n)\big )$ and $\mathbb P_{(1,\ldots ,1)} \big (\mathbb A^{(\l)}(\mathbb  D^n)\big )$ are locally free of rank $1$ on $\mathbb G_n.$
\end{cor}

\begin{lem}\label{coe2}
Let ${ K}^{(\l)}_{\mb G_n}$ be the  normalized reproducing kernel for $\m H^{(\l)}(\mb{G}_n).$ Then
\begin{enumerate}
\item[(i)]  the coefficient of  $s_1(\bl z)\overline{s_1(\bl w)}$  in  $K_{\mb{G}_n}^{(\l)}\big(\bl s(\bl z ),\bl s(\bl w)\big)$ is $ \frac{\l}{n},$
\item[(ii)]  the coefficient of  $s_1(\bl z)^2\overline{s_1(\bl w)^2}$  in  $K_{\mb{G}_n}^{(\l)}\big(\bl s(\bl z ),\bl s(\bl w)\big)$ is $\frac{\l(\l +1)}{2n}.$
\end{enumerate}
\end{lem}
\begin{proof}

Since the monomial symmetric polynomial  $\mb M_{\bl m}$ is a homogeneous  symmetric polynomial of degree $\vert \bl m\vert:=\sum_{i=1}^nm_i ,$  therefore, it is a polynomial in the elementary symmetric polynomials $s_i(\bl z)$ for $i=1,\ldots, n$. For a fixed $k, q \in \mb{Z}_{+} ,$ the term $s_k(\bl z)^q\overline{s_k(\bl w)^q}$ in $K_{\mb{G}_n}^{(\l)}\big(\bl s(\bl z),\bl s(\bl w)\big)$ comes only from the terms involving   $\mb{M}_{\bl m}(\bl z)\overline{\mb{M}_{\bl m}(\bl w)}$ in the series for  $K_{\mb{G}_n}^{(\l)}\big(\bl s(\bl z),\bl s(\bl w)\big)$ in   Lemma \ref{km}, where   $\bl m=(m_1,\ldots,m_n) \in [n]$ such that $\sum_{i=1}^{n} m_i = kq.$

To obtain the coefficient of $s_1(\bl z)\overline{s_1(\bl w)}$ in $K_{\mb{G}_n}^{(\l)}\big(\bl s(\bl z),\bl s(\bl w)\big),$ we only need to consider the term $\mb{M}_{\bl m}(\bl z)\overline{\mb{M}_{\bl m}(\bl w)},$ for $\bl m=(1,0,...,0).$
Note that $\mb{M}_{\bl m}(\bl z) = s_1(\bl z).$ Since $\bl m=(1,0,\ldots ,0)$ has multiplicity $\alpha = (1,(n-1)),$ it follows that  the coefficient of  $s_1(\bl z)\overline{s_1(\bl w)}$  in $ K_{\mb{G}_n}^{(\l)}\big(\bl s(\bl z),\bl s(\bl w)\big) $ is
\Bea
 \frac{1}{n!}\cdot \frac{{\bl\alpha!(\l)_{\bl m}}}{\bl m!}= \frac{(n-1)!1!(\l)_1}{1!n!} = \frac{\l}{n}.
\Eea
This proves (i).

Analogously, to find the coefficient of $s_1(\bl z)^2\overline{s_1(\bl w)^2}$ in $K_{\mb{G}_n}^{(\l)}\big(\bl s(\bl z),\bl s(\bl w)\big),$  we need to consider terms corresponding to  $\bl m=(2,0,...,0)$ and $\bl m=(1,1,0,...,0).$ Note that  $\mb{M}_{\bl m}(\bl z)=s_2(\bl z)$ for  $\bl m=(1 ,1, 0, \ldots ,0),$  so the  coefficient of the term   $\mb{M}_{\bl m}(\bl z)\ov{\mb M_{\bl m}(\bl w)}$ for  $\bl m=(1 ,1, 0, \ldots ,0),$ will not contribute here. Now  $\mb{M}_{\bl m}(\bl z)=s_1(\bl z)^2- 2 s_2(\bl z)$  for $\bl m=(2, 0, \ldots, 0).$ Since $\bl m=(2,0,\ldots,0)$ has  multiplicity $\bl\a=(1, n-1),$ it follows that the  coefficient of  $s_1(\bl z)^2\overline{s_1(\bl w)^2}$  in $ K_{\mb{G}_n}^{(\l)}\big(\bl s(\bl z),\bl s(\bl w)\big) $ is

\Bea
 \frac{1}{n!}\cdot \frac{{\bl\alpha!(\l)_{\bl m}}}{\bl m!}\nonumber=\frac{(n-1)! (\l)_2}{2!n!}=\frac{\l(\l+1)}{2n}.
\Eea
This proves (ii).
\end{proof}

\begin{proof}[Proof of Theorem \ref{Gn}]
 If possible, let these two modules be unitarily equivalent. Recall that the reproducing kernels  $\wi{\mathbf B}_{\mb{G}_n}^{(\l)}$ and $K_{\mb{G}_n}^{(\l)}$  have the property that
\Bea
\wi{\mathbf B}_{\mb{G}_n}^{(\l)}\big(\bl s(\bl z),0\big)= K_{\mb{G}_n}^{(\l)}\big(\bl s(\bl z),0\big)=1\,\, \text{ for } \bl s(\bl z)\in\mb G_n,
\Eea
that is, these are the normalized reproducing kernels at $0$ of the respective Hilbert spaces.
Since by construction, the polynomial ring $\mb C[s_1,\ldots,s_n]=\mb C[\bl z]^{\mathfrak S_n}$ in $n$ variables is dense in both $\m H^{(\l)}(\mb G_n)$ and $\mb A^{(\l)}(\mb G_n),$ it follows (cf. \cite[Remark, p. 285]{DM}) that
the dimension of the joint kernel  is $1$ for all $\bl w \in \mathbb G_n.$
Therefore, by   \cite[Lemma 4.8(c)]{CS},  we infer that
\Bea
\wi{\mathbf B}_{\mb{G}_n}^{(\l)}\big(\bl s(\bl z),\bl s(\bl w)\big)= K_{\mb{G}_n}^{(\l)}\big(\bl s(\bl z),\bl s(\bl w)\big)\,\, \text{ for } \bl s(\bl z),\bl s(\bl w)\in\mb G_n.
\Eea

Equating the coefficients of $s_1(\bl z)\ov{s_1(\bl z)}$
from  Lemma \ref{coe1}
we see that $\l = \l + n - 1.$ Thus we must have $n=1$ completing the proof of  the Theorem.
\end{proof}
\begin{cor}\label{corf}
In the decomposition of the Hilbert module $\mathbb A^{(\l)}(\mathbb D^3):$
$$\mathbb A^{(\l)}(\mathbb D^3)= \mb P_{(3)}\big (\!\mb A^{(\l)}(\mb D^3)\!\big ) \oplus \mb P_{(2,1)}\big (\mb A^{(\l)}(\mb D^3)\big )\oplus \mb P_{(1,1,1)}\big (\mb A^{(\l)}(\mb D^3)\big ),$$ all the sub-modules  on the right hand side of the equality are inequivalent.
\end{cor}
\begin{proof}
We have just proved that $\mb P_{(3)}\big (\mb A^{(\l)}(\mb D^3)\big )$ cannot be equivalent to $\mb P_{(1,1,1)}\big (\mb A^{(\l)}(\mb D^3)\big ),$ in general. Since the rank of the sub-module $\mb P_{(2,1)}\big (\mb A^{(\l)}(\mb D^3)\big )$ is $\chi_{(2,1)}(1)^2=4,$ \cite[Example 2.6]{FH}, it cannot be equivalent to either of these.
\end{proof}

%
%
\begin{rem} The proof of Theorem \ref{m} shows that we have proved a little more than what is claimed in the Theorem, namely:
The Hilbert modules $\mb A^{(\l)}_{\rm sym}(\mb D^n)$ and $\mb A^{(\mu)}_{\rm anti}(\mb D^n)$ over $\mb C[\bl z]^{\mathfrak S_n}$ are not  equivalent for any $\l, \mu > 0$ and $n \geq 2.$
To prove this more general claim, we merely note, as before,  that
equating the coefficients of $s_1(\bl z)\ov{s_1(\bl z)}$ and $s_1(\bl z)^2\ov{s_1(\bl z)^2}$ from  Lemma \ref{coe1} and Lemma \ref{coe2}, we obtain
\Bea\label{oneequate}
\l = \mu + n - 1  &\text{ and }& \frac{\l(\l+1)}{2n}=\frac{(\mu + n-1)(\mu + n)}{n(n+1)}.
\Eea
Combining these equations,
we have that $n=1,$ which proves our claim.  Indeed, the two modules $\mathbb P_{\bl p}^{ii} \big (\mathbb A^{(\l)}(\mathbb  D^n)\big )$ and $\mathbb P_{\bl q}^{jj}\big ( \mathbb A^{(\mu)}(\mathbb  D^n)\big )$ are not equivalent either for any $1 \leq i \leq \chi_{\bl p}(1)$ and $1 \leq j \leq \chi_{\bl q}(1)$ for which $\chi_{\bl p}(1)\ne \chi_{\bl q}(1).$
\end{rem}

In cases where $\chi_{\bl p}(1) > 1$, we believe, the work of \cite{KZ} and \cite{ES},
may be useful in answering the question of mutual equivalence of the sub-modules $\mathbb P_{\bl p}^{ii}\big ( \mathbb A^{(\l)}(\mathbb  D^n)\big )$. We intend to explore this possibility  in our future work.

Let $\mathcal H$ be a locally free Hilbert module over $\Omega\subseteq \C^n$. Following \cite{KZ} and \cite{ES}, we define a holomorphic  section $\gamma:\Omega\ra\mathcal H$ to be a \emph{spanning holomorphic cross-section} for $\mathcal H$ if
$$
\bigvee\{\gamma(z) : z\in\Omega\} = \mathcal H.
$$

Building on the work in \cite{KZ}, the existence of a spanning holomorphic cross-section for a large class of Hilbert modules over an admissible set was proved in \cite{ES}. However, in the case of the sub-modules $\mb P_{\bl p}^{ii}\big(\mb A^{(\l)}(\mb D^n)\big),$ the existence of a spanning holomorphic cross-section is easily established by exhibiting such a section. Indeed, we give an explicit realization of the spanning holomorphic cross-section for these sub-modules.

Let  $U$ be an open neighbourhood of of $\bl u_0$ in $\big(\mb G_n\setminus\bl s(\mathcal Z)\big)\cap \bl s\big(\mb D^n\setminus X\big).$ The function $\bl s$ admits $n!$ local inverses on the open set $U.$ Fix one such, say $\phi$. Define $\gamma(\bl u) =\mb P_{\bl p}^{ii} K^{(\l)}(\cdot, \phi(\bar{\bl u})),$  $\bl u \in U^*$. From Equation \eqref{pii}, it follows that $\gamma$ is a spanning holomorphic cross-section for $\mb P_{\bl p}^{ii}\big(\mb A^{(\l)}(\mb D^n)\big)$. Let $E_{\bl p}^{(i)} = \{(\bl u, x)\in U^*\times\mb P_{\bl p}^{ii}\big(\mb A^{(\l)}(\mb D^n)\big)\mid x = c\gamma(\bl u) \mbox{~for~some~} c\in\C \}$ denote the corresponding holomorphic hermitian  line bundle and
$$
\mathscr K_{\bl p}^{(i)}(\bl u) = -\sum_{j,k = 1}^n\partial_j\bar{\partial}_k \log \|\gamma(\bl u)\|^2 du_j\wedge d\bar {u_k},
$$
be the curvature of $E_{\bl p}^{(i)}.$ Now, we restate Theorem 5.2 of \cite{ES} using the spanning cross-sections we have found here.
\begin{thm}
For any two partitions  $\bl p$ and $\bl q$ of $n$ and for $i,j,$ $1\leq i\leq\chi_{\bl p}(1),$ $1\leq j\leq\chi_{\bl q}(1),$  the sub-modules  $\mb P_{\bl p}^{ii}\big(\mb A^{(\l)}(\mb D^n)\big)$ and $\mb P_{\bl q}^{jj}\big(\mb A^{(\l)}(\mb D^n)\big)$ are equivalent if and only if $\mathscr K_{\bl p}^{(i)} = \mathscr K_{\bl q}^{(j)}$.
\end{thm}


\end{document}